\documentclass{amsart}
\usepackage[all]{xy}
\usepackage[pagebackref,colorlinks=true,linkcolor=blue,urlcolor=blue]{hyperref}
\usepackage{color}
\usepackage{cleveref}

\newtheorem*{theorem*}{Theorem}
\newtheorem{theorem}{Theorem}[section]
\newtheorem{lemma}[theorem]{Lemma}
\newtheorem{proposition}[theorem]{Proposition}
\newtheorem{corollary}[theorem]{Corollary}

\newtheorem{headthm}{Theorem}

\theoremstyle{definition}
\newtheorem{definition}[theorem]{Definition}

\newtheorem{setup}[theorem]{Setup}
\newtheorem{assumption}[theorem]{Assumption}

\theoremstyle{remark}
\newtheorem{remark}[theorem]{Remark}

\newcommand{\Tor}{\textup{Tor}}
\newcommand{\Proj}{\textup{Proj}}
\newcommand{\depth}{\textup{depth}}
\newcommand{\Tr}{\textup{Tr}}
\newcommand{\Hom}{\textup{Hom}}
\newcommand{\Ext}{\textup{Ext}}

\newcommand{\Ass}{\textup{Ass}}

\renewcommand\rightmark{Local cohomology tables of sequentially almost Cohen-Macaulay modules}

\numberwithin{equation}{section}

\begin{document}


\title{Local cohomology tables of sequentially almost Cohen-Macaulay modules}
\author{Cheng Meng}

\address{Cheng Meng\\ Yau Mathematical Sciences Center, Tsinghua University, Beijing 100084, China. \emph{Email:} {\rm cheng319000@tsinghua.edu.cn}}

\date{\today}
\maketitle

\begin{abstract}
Let $R$ be a polynomial ring over a field. We introduce the concept of sequentially almost Cohen-Macaulay modules and describe the extremal rays of the cone of local cohomology tables of finitely generated graded $R$-modules which are sequentially almost Cohen-Macaulay, and describe some cases when the local cohomology table of a module of dimension 3 has a nontrivial decomposition.
\end{abstract}

\section{introduction}
Let $R = k[x_1,\ldots,x_n]$ be a standard graded polynomial ring over a field $k$. The graded Betti numbers and the local cohomology modules are important homological data of graded modules over $R$. In 2006, Boij and S\"oderberg \cite{boij2008graded} formulated two conjectures on the cone of graded Betti tables of finitely generated Cohen-Macaulay modules, which were proved by David Eisenbud, Gunnar Fl{\o}ystad and Jerzy Weyman in characteristic 0 in \cite{eisenbud2011existence} and by Eisenbud and Schreyer in arbitrary characteristic in \cite{eisenbud2009betti}. These conjectures were also extended to the non-Cohen-Macaulay case by Boij and S\"oderberg in \cite{boij2012betti}.

Denote the Betti table of a finitely generated graded module by $\beta^{\bullet}(\cdot)$, then the above results can be restated in the following way:

\begin{theorem}\label{1.1}
Let $R = k[x_1,\ldots,x_n]$ be a standard graded polynomial ring. The extremal rays of the cone generated by Betti tables of finitely generated graded $R$-modules are given by the modules with a pure resolution, and every Betti table in the cone decomposes in the following way. For every finitely generated graded module $M$ there exist finitely generated modules $N_1,\ldots,N_s$ with pure resolutions and $r_1,\ldots,r_s \in \mathbb{Q}$, $r_1,\ldots,r_s>0$ such that
$$\beta^{\bullet}(M) = \sum_{i = 1}^s r_i\beta^{\bullet}(N_i).$$
\end{theorem}
Eisenbud and Schreyer, in \cite{eisenbud2009betti}, asked for a similar description for the cone of cohomology tables of coherent sheaves, and they proved a result similar to Theorem \ref{1.1} in \cite{eisenbud2010cohomology}:

\begin{theorem}
Let $R = k[x_1,\ldots,x_n]$ be a standard graded polynomial ring, $X=\Proj (R)$. The extremal rays of the cone generated by cohomology tables of coherent sheaves are given by those of supernatural vector bundles, and for every coherent sheaf $\mathcal{F}$  there exist possibly infinitely many supernatural vector bundles $\mathcal{F}_1,\mathcal{F}_2,\ldots$ and $r_1,r_2,\ldots \in \mathbb{Q}$, $r_i>0$ such that
$$H^{\bullet}(X,\mathcal{F}) = \sum_{i \geq 1} r_iH^{\bullet}(X,\mathcal{F}_i).$$
\end{theorem}
Motivated by the above results, Daniel Erman asked whether one could describe the cone of local cohomology tables of finitely generated graded $R$-modules similarly. The theme of this paper is to give answers to this question for various kinds of cones. This problem on the cone of local cohomology tables arises formally as Question 2.1 in Smirnov and De Stefani's paper \cite{de2021decomposition}: can we find a description of the extremal rays of the cone of local cohomology tables, and do the extremal rays generate all tables in this cone? Although the local cohomology modules and the sheaf cohomology are closely related, Eisenbud and Schreyer's result does not directly lead to an immediate answer due to two obstacles. The first obstacle is about finiteness, since in Question 2.1 of \cite{de2021decomposition} we require the linear combination to be a \emph{finite} sum, while Eisenbud and Schreyer's result uses a convergent \emph{infinite} sum. The second obstacle is that the local cohomology tables of modules and cohomology tables of coherent sheaves differ at the beginning. For instance, let $X=\Proj (R)$, $M$ be a finitely generated graded $R$-module, $\mathfrak{m}=(x_1,\ldots,x_n)$ be the graded maximal ideal, $\mathcal{F}=\tilde{M}$ be the corresponding coherent sheaf. Let $\Gamma(M)=\oplus_{t \in \mathbb{Z}}H^{0}(X, \mathcal{F}(t))$ be the \emph{module of global sections of M}. Then $H^i(X,\mathcal{F})=H^{i+1}_{\mathfrak{m}}(M)$ for $i \geq 1$ and there is an exact sequence
\begin{equation}\label{equation-relation of Gamma}
0 \to H^0_{\mathfrak{m}}(M) \to M \to \Gamma(M) \to H^1_{\mathfrak{m}}(M) \to 0.    
\end{equation}
The sheaf $\mathcal{F}$ only determines $\Gamma(M)$, and in general it is hard to find a decomposition of $H^1_{\mathfrak{m}}(M)$ using the data of $\Gamma(M)$.

This problem is still open in general. Up to now, there are two important results of this kind. One result is proved by Smirnov and De Stefani in \cite{de2021decomposition}, where they gave a complete description of the cone of local cohomology tables of modules of dimension at most 2. By denoting the local cohomology table of a graded module by $H^{\bullet}(\cdot)$, Smirnov and De Stefani's results are as follows:

\begin{theorem}[\cite{de2021decomposition}, Theorem 4.6]\label{theorem-S-DS main result}
Let $R = k[x_1,\ldots,x_n]$ and $S = k[x_1,x_2]$ be two standard graded polynomial rings. Assume $n \geq 2$. Let $A$ be the set of $S$-modules $\{k(a),k[x](a), S(a), (x_1,x_2)^t(a), t \in \mathbb{N},a \in \mathbb{Z}\}$. Identify these $S$-modules as $R$-modules via the ring map $S \cong R/(x_3,\ldots,x_n)$. Then for every finitely generated graded $R$-module $M$ of dimension at most $2$, there exist $N_1,\ldots,N_s \in A$, $r_1,\ldots,r_s \in \mathbb{Q}$ and $r_1,\ldots,r_s>0$ such that
$$H^{\bullet}(M)=\sum_{i=1}^s r_iH^{\bullet}(N_i).$$
Moreover, the set $A$ describes the vertex set of the cone of local cohomology tables of finitely generated graded modules of dimension at most $2$.
\end{theorem}
Later, Caviglia and De Stefani proved a similar result in \cite{CAVIGLIA2021106635}; they define a concept called $E$-depth which relies on the depths of certain Ext modules. They gave a description of the vertex set of the cone of local cohomology tables of modules with $E$-depth at least $n-2$, and proved that these vertices generate every table in this cone. Their result reduces to Theorem \ref{theorem-S-DS main result} when $n=2$.

In this paper we will analyze this problem in a different setting from the one considered in \cite{CAVIGLIA2021106635} and \cite{de2021decomposition}. First, we consider a special class of modules called sequentially almost Cohen-Macaulay(saCM) modules (see Definition \ref{definition-sacm}), and discuss how the local cohomology tables of such modules decompose. The following theorem is the first main result of this article:
\begin{headthm}[Theorem \ref{4.19}]\label{theorem A}Let $R = k[x_1,\ldots,x_n]$ be a standard graded polynomial ring. Let $C$ be the cone generated by local cohomology tables of saCM modules. The vertex set of this cone is given by the set $A_M$ (see Definition \ref{definition-AM,A'M}); there exists a set $A'_M$ (see Definition \ref{definition-AM,A'M}) properly containing $A_M$ such that for every local cohomology table $H^{\bullet}(M)$ in the cone, there is a finite set of modules $N_1,\ldots,N_s \in A'_M$ and $r_1,\ldots,r_s \in \mathbb{Q}$ and $r_1,\ldots,r_s>0$ such that
$$H^{\bullet}(M)=\sum_{i=1}^s r_iH^{\bullet}(N_i).$$
However, when $n \geq 3$, then there is an saCM module $M_0$ such that for any finite set of saCM modules $N_1,\ldots,N_s \in A_M$ and $r_1,\ldots,r_s \in \mathbb{Q}$ and $r_1,\ldots,r_s>0$,
$$H^{\bullet}(M) \neq \sum_{i=1}^s r_iH^{\bullet}(N_i).$$
That is, $A_M$ does not generate the cone.
\end{headthm}
The descriptions of the sets $A_M,A'_M$ in Definition \ref{definition-AM,A'M} rely on the concept of the minimal Auslander transpose $\Tr(\bullet)$, see Definition \ref{Sec4-auslander transpose}. The main idea is that under proper assumptions on $M$ (see Assumption \ref{assumptions-Mpd1,NtrM}), there is a $\mathbb{Q}$-linear transformation that maps $\beta^{\bullet}(M)$ to $H^{\bullet}(\Tr(M))$. The existence of such a transformation allows us to use the Boij-S\"oderberg theory of Betti tables to decompose a local cohomology table. The vertices of the cone of Betti tables are of the form $\beta^{\bullet}(N)$ where $N$ has a pure resolution and every such table can be computed using the degree sequence of $N$, so we only need to compute $H^{\bullet}(M)$ when $N=\Tr(M)$ has a pure resolution, and determine when such a local cohomology table decomposes.

Theorem \ref{theorem A} only talks about the decomposition of local cohomology table of saCM modules. There is another important result on the cone of local cohomology tables of all finitely generated graded modules:
\begin{headthm}[Theorem \ref{theorem-all local cohomology tables}]
Let $R$ be a polynomial ring of dimension $n \geq 3$. Let $C_H$ be the cone of local cohomology tables of all finitely generated graded $R$-modules. Then $C_H$ is not generated by its vertices.    
\end{headthm}
This answers the second part of Smirnov and De Stefani's question in negative for $n \geq 3$, thus shows that the largest dimension of $R$ where the second part holds true is exactly 2. This also reflects the difference between the results in this paper and the results in \cite{de2021decomposition} or \cite{CAVIGLIA2021106635} where the cones are generated by their vertices.

Then, we study the decomposition of local cohomology tables of modules of dimension 3. In this case we may also assume that $\dim(R)=3$ by Lemma \ref{3.7}, and we study when the local cohomology $H^\bullet(M)$ is decomposable, which indicates the necessary conditions for the extremal rays of the cone of local cohomology tables. We first reduce to the case where depth$(M)=1$ and $M$ has no dimension 1 submodule. Then we relate such $M$ to the module of global sections $\Gamma(M)$; in this case $\Gamma(M)$ is finitely generated with depth$(\Gamma(M)) \geq 2$. This means $\Gamma(M) $ is Cohen-Macaulay or almost Cohen-Macaulay, so it is saCM and its cohomology table decomposes according to Theorem \ref{4.19}. The key point is whether a decomposition of $H^{\bullet}(\Gamma(M))$ induces a decomposition of $H^{\bullet}(M)$, and it does in two cases, described by the following two theorems. In the first case, there is a submodule of dimension 2 of $M$ that induces a decomposition:
\begin{headthm}[Theorem \ref{6.8}]
Let $M$ be a module of depth $1$ and assume $M$ has no dimension $1$ submodule. Let $\Gamma=\Gamma(M)$ (see Section \ref{section-Gamma functor}), $Q=\Gamma/\Tor(\Gamma)$. Then 
$$H^{\bullet}(M) = H^{\bullet}(\Tor(M))+H^{\bullet}(M/\Tor(M))-(0,HS(H^1_{\mathfrak{m}}(Q)),HS(H^1_{\mathfrak{m}}(Q)),0).$$ 
In particular, if $H^1_{\mathfrak{m}}(Q)=0$, then 
$$H^{\bullet}(M) = H^{\bullet}(\Tor(M))+H^{\bullet}(M/\Tor(M)).$$
\end{headthm}
Here $HS(\bullet)$ refers to the Hilbert series, and a description of the module $H^1_{\mathfrak{m}}(Q)$ is given in Proposition \ref{6.10}. Note that here we view the vector with power series entries $(0,HS(H^1_{\mathfrak{m}}(Q)),HS(H^1_{\mathfrak{m}}(Q)),0)$ as a table; this is called a table in series form, which will be explained in Section \ref{section-notations}.

In the second case there is a submodule of dimension 3 of $M$ that induces a decomposition:
\begin{headthm}[Theorem \ref{6.14}]
Let $M$ be a module of depth $1$ and assume $M$ has no dimension $1$ submodule. Let $\Gamma=\Gamma(M)$ (see Section \ref{section-Gamma functor}). Suppose $\Ext^2_R(\Tr(\Gamma),R)=0$ and $\Gamma^*=\Hom_R(\Gamma,R) \neq \Tr(L')$ for any module $L'$ of finite length. Then $H^{\bullet}(M) = H^{\bullet}(M \cap F)+H^{\bullet}(M/M \cap F)$ for some free module $F \subset \Gamma$.
\end{headthm}

\subsection*{Outline}
The structure of the paper is as follows.  In Section \ref{section-notations}, we define a table in series form, review some basic definitions on convex cones in a $\mathbb{Q}$-vector space, and some other basic propositions. Section \ref{section-dimension filtration} covers the concept of the dimension filtration introduced by Schenzel, and shows that to decompose the local cohomology table of an saCM module, it suffices to decompose the tables of modules of projective dimension at most 1; Section \ref{section-pd<=1case} shows how to decompose these tables. In Section \ref{section-Gamma functor}, we introduce some properties of the $\Gamma$-functor and prove that in order to decompose $H^{\bullet}(M)$, it suffices to decompose $H^{\bullet}(\Gamma(M))$ and $HS(\Gamma(M)/M)$ simultaneously. Note that if the module has dimension 3, then we reduce to the case where the ring has dimension 3; in this case the projective dimension of $\Gamma(M)$ is at most 1. Finally in Section \ref{section-decomposition dim3}, we find conditions under which we can decompose $H^{\bullet}(\Gamma(M))$ and $HS(\Gamma(M)/M)$ simultaneously.
\section{notations}\label{section-notations}
Throughout the paper, we work in the following setup: 
\begin{setup}
Let $k$ be a field, $R=k[x_1,\ldots,x_n]$ be a standard graded polynomial ring over a field $k$, $\mathfrak{m}=(x_1,\ldots,x_n)$ be its graded maximal ideal. We assume $M$ is a finitely generated graded $R$-module.    
\end{setup}
\subsection{Notations on the tables}
Conventionally, the \emph{Betti table} of $M$ is the $(n+1) \times \mathbb{Z}$-table $\beta^{\bullet}(M)$ with entries $\beta^{\bullet}(M)_{i,j} = \dim_k\Tor^R_i(M,k)_{i+j}$. The \emph{local cohomology table} $H^{\bullet}(M)$, by definition, is the $(n+1) \times \mathbb{Z}$-table defined by $H^{\bullet}(M)_{i,j} = \dim_kH^i_{\mathfrak{m}}(M)_j$. We define the \emph{Ext-table} $E^{\bullet}(M)$ to be the $(n+1) \times \mathbb{Z}$-table where $E^{\bullet}(M)_{i,j} = \dim_k\Ext^i_R(M,R)_j$. 

\begin{table}[h!]
 \centering
\begin{tabular}{|c||c|c|c|c|}
\hline
   & 0 & 1 & 2 & 3\\
\hline\hline
 0 & 1 & 0 & 0 & 0\\
 \hline
 1 & 0 & 3 & 3 & 1\\
 \hline
 2 & 0 & 1 & 1 & 0\\
\hline
\end{tabular}
\caption{The Betti table of $R/(x_1^2,x_1x_2,x_1x_3,x_2^3)$}
\label{Table 1}
\end{table}

Table \ref{Table 1} corresponds to the minimal free resolution of the module $R/(x_1^2,x_1x_2,x_1x_3,x_2^3)$ over $R$ when $n \geq 3$, which is
$$0 \to R(-4) \to R(-4)\oplus R(-3)^{3} \to R(-3)\oplus R(-2)^{3} \to R.$$
Throughout this paper we use a different kind of notations, and the tables considered will be in \emph{series form}. To be precise, we make the following definition:
\begin{definition}
\phantom{break}
\begin{enumerate}
\item We say the space of Betti tables or Ext-tables is $V=\oplus^n_{i=0}\mathbb{Q}[[t]][t^{-1}]v_i$. It is a free $\mathbb{Q}[[t]][t^{-1}]$-module of rank $n+1$, and it is also a $\mathbb{Q}$-vector space.
\item The space of local cohomology tables is $V^*=\oplus^n_{i=0}\mathbb{Q}[[t^{-1}]][t]v^*_i$.
\item The unshifted Betti table of $M$ in series form is an element $\beta^\bullet(M)=(\beta_0(M),\beta_1(M),\ldots,\beta_n(M)) \in V$ defined by $\beta_i(M)=\Sigma_{j \in \mathbb{Z}}\beta_{i,j}(M)t^j$.
\item The Ext-table in series form is an element $E^\bullet(M)=(E^0(M),E^1(M)$, \ldots,$E^n(M)) \in V$ where $E^i(M)=\Sigma_{j \in \mathbb{Z}}$dim$_k\Ext^i_R(M,R)_jt^j$.
\item The local cohomology table of $M$ is $H^\bullet(M)=(h^0(M),h^1(M),\ldots,h^n(M))$ $\in V^*$ where $h^i(M)=\Sigma_{j \in \mathbb{Z}}$dim$_kH^i_{\mathfrak{m}}(M)_jt^j$.
\end{enumerate}    
\end{definition}
These two representations of a table are equivalent. For example, Table \ref{Table 1} will become $(1,3t^2+t^3,3t^3+t^4,t^4)$ in series form.
\begin{remark}
In all of the following sections, we will stick to the \emph{unshifted} Betti table instead of the usual Betti table, and we will express all the tables in series form. The series form has two advantages: first, the entries of these tables are series, and they interact with the Hilbert series of graded modules; and second, the action of taking the difference of a table becomes multiplication by $(1-t)$, which makes sense as $V,V^*$ are $\mathbb{Q}[t]$-modules.   
\end{remark}
For a graded module $M$, let $HS(M)=\sum_{i \in \mathbb{Z}}\dim_k(M_i)t^i$ denote its Hilbert series. One interaction of the tables with the Hilbert series is reflected in the following proposition.
\begin{proposition}
Let $\beta^\bullet(M)=(\beta_0(M),\ldots,\beta_n(M))$ be the Betti table of $M$, then
$$HS(M)=\sum_{0 \leq i \leq n}(-1)^i\beta_i(M)HS(R)=(1-t)^{-n}\sum_{0 \leq i \leq n}(-1)^i\beta_i(M).$$
\end{proposition}
\begin{proof}
For fixed $i$, $\beta_i(M)HS(R)=HS(F_i)$ where $F_i$ is the $i$-th free module in the free resolution of $i$. Since $M$ is the homology of the free resolution and taking homology does not change the alternating sum of Hilbert series, we are done.    
\end{proof}

\subsection{Convex cones in a $\mathbb{Q}$-vector space}
We want to consider \emph{convex cones} $C$ in the vector space $V$ or $V^*$, that is, subsets that are closed under multiplication by positive rational numbers and addition. We call the expression $\sum_{1 \leq i \leq s}a_ic_i$ with $c_i \in C, a_i \in \mathbb{Q}, a_i > 0$ a \emph{positive linear combination} of $c_1,c_2,\ldots,c_s$. If $c \in C$ is a positive linear combination of $c_1,c_2,\ldots,c_s$, we also say that $c$ \emph{decomposes into $c_1,c_2,\ldots,c_s$}; we say the decomposition is trivial if $s=1$ and in this case $c$ and $c_1$ differ by a positive rational scalar. A \emph{generating set} of the cone is a subset $G$ of the cone $C$ such that every element is a positive linear combination of elements in $G$. We also say $G$ generates the cone $C$ if $G$ is a generating set. A \emph{vertex} is an element that does not decompose nontrivially and the \emph{vertex set} is the set of all vertices. We say that a ray inside the cone is \emph{extremal} if it contains a vertex; in this case every element of the ray is a vertex except for the origin. In this paper, we will consider the vertex sets of 3 kinds of cones: the cones generated by the Betti tables, the Ext-tables, and the local cohomology tables. It is easy to see that if $G$ is a generating set and $v$ is a vertex, then $v$ must decompose trivially, which means that a positive multiple of $v$ is in $G$. So to find the vertex set, we may find a generating set $G$ first and then find elements in $G$ that decompose trivially. The following lemma about cones is trivial but will be useful in Section \ref{section-pd<=1case}.
\begin{lemma}\label{2.1}Let $L:W_1 \to W_2$ be a linear map between vector spaces over $\mathbb{Q}$. Suppose $C \subset W_1$ is a cone with vertex set $V_1$ and a generating set $G_1$. Then $L(C)$ is a cone in $W_2$ generated by $L(G_1)$; suppose the vertex set of this cone is $V_2$, then $V_2 \subset L(V_1)$. Furthermore, $V_2$ is also the subset of $L(V_1)$ which is not a positive linear combination of the other elements in $L(G_1)$. If moreover, $L$ is an injection, then $V_2 = L(V_1)$.
\end{lemma}

\subsection{Some Lemmas on tables in series form}
Let $E=E_R(k)$ be the graded injective hull of $k=R/\mathfrak{m}$. Recall that by local duality, 
$$H^i_{\mathfrak{m}}(M) = \textup{Hom}_R(\Ext^{n-i}_R(M,R(-n)),E).$$ 
This implies dim$_k(H^i_{\mathfrak{m}}(M)_j)=$
dim$_k(\Ext^{n-i}_R(M,R)_{-n-j})$, hence we have:
\begin{proposition}The $\mathbb{Q}$-linear map $L_0:V \to V^*$, where
$$L_0(f_0(t),f_1(t),\ldots,f_n(t))=t^{-n}(f_n(t^{-1}),f_{n-1}(t^{-1}),\ldots,f_0(t^{-1})),$$
is invertible and $H^{\bullet}(M)=L_0(E^{\bullet}(M))$.
\end{proposition}
By the above proposition, the extremal rays of the cone of local cohomology tables and the cone of Ext-tables are in 1-1 correspondence under $L_0$. So to find the extremal rays of the cone generated by all local cohomology tables, it suffices to find those of all Ext-tables.

It is well known that the Betti numbers are nonzero for finitely many entries, and the dimension of the $i$-th Ext module is at most $n-i$. So actually these tables sit in some proper subspaces of $V$ or $V^*$ respectively.
\begin{proposition}The following proposition holds.
\item (1) $\beta^{\bullet}(M) \in \oplus^n_{i=0}\mathbb{Q}[t][t^{-1}]v_i$.

\item (2) $E^{\bullet}(M) \in \oplus^n_{i=0}\mathbb{Q}[t][t^{-1}]\frac{1}{(1-t)^{n-i}}v_i$.

\item (3) $H^{\bullet}(M) \in \oplus^n_{i=0}\mathbb{Q}[t^{-1}][t]\frac{1}{(1-t^{-1})^i}v_i$.
\end{proposition}
Suppose we have a short exact sequence of finitely generated graded $R$-modules $0 \to M' \to M \to M'' \to 0$. Then we have a long exact sequence of local cohomology modules. We see from the long exact sequence that $H^{\bullet}(M)=H^{\bullet}(M')+H^{\bullet}(M'')$ if and only if all the connecting maps $H^{i}_{\mathfrak{m}}(M'') \to H^{i+1}_{\mathfrak{m}}(M')$ are 0.
\begin{definition}\label{definition-exact sequence induce decomp}
We say that the exact sequence $0 \to M' \to M \to M'' \to 0$ induces a decomposition of local cohomology tables, if $H^{\bullet}(M)=H^{\bullet}(M')+H^{\bullet}(M'')$ holds.   
\end{definition}
Finally, the depths of these modules are related by the well-known depth lemma:
\begin{proposition}[Depth lemma]Let $0 \to M' \to M \to M'' \to 0$ be an exact sequence of finitely generated $R$-modules, then:
\item (1) $\depth(M) \geq \textup{min}\{\depth(M'), \depth(M'') \}$.

\item (2) $\depth(M') \geq \textup{min}\{\depth(M), \depth(M'')+1 \}$.

\item (3) $\depth(M'') \geq \textup{min}\{\depth(M')-1, \depth(M) \}$.
\end{proposition}

\section{The dimension filtration}\label{section-dimension filtration}
Let us recall the concept of the dimension filtration introduced by Schenzel in \cite{schenzel1999dimension}.

\begin{definition}
Let $A$ be a Noetherian ring of dimension $d$, and $M$ be a finitely generated $A$-module.
\begin{enumerate}
\item Suppose $M_i$ is the largest submodule of $M$ such that $\dim(M_i) \leq i$.  Then We say $M_i$ is the largest submodule of $M$ of dimension at most $i$; if it happens that $\dim(M_i)=i$ we say it is the largest submodule of $M$ of dimension $i$.
\item We have that $0 \subset M_0 \subset M_1 \subset M_2 \subset \ldots \subset M_d=M$ forms a filtration of $M$, called the dimension filtration of $M$.
\item We say $M_i/M_{i-1}$ the $i$-th dimension factor of $M$.
\end{enumerate}
\end{definition}
\begin{remark}
For any $i$, $M_i$ exists by the Noetherian property, since the sum of two modules of dimension at most $i$ has dimension at most $i$. Each dimension factor $M_i/M_{i-1}$ is either 0 or of dimension $i$, and $M$ has no nonzero submodule of dimension $i$ if and only if the $i$-th dimension factor is 0. We set $M_{-1}=0$ so that the maximal submodule of dimension $0$ is also the $0$-th dimension factor.    
\end{remark}
In the following part of the paper, we will fix the notation $M_i$ for the dimension filtration instead of referring to other modules.

\begin{proposition}For any $M$, the following holds:
\item (1) $\Ass(M_i)=\{\mathfrak{p} \in \Ass(M)|\dim(A/\mathfrak{p}) \leq i\}$.

\item (2) $\Ass(M/M_i)=\{\mathfrak{p} \in \Ass(M)|\dim(A/\mathfrak{p}) > i\}$.

\item (3) $\Ass(M_i/M_{i-1})=\{\mathfrak{p} \in \Ass(M)|\dim(A/\mathfrak{p}) = i\}$.
\end{proposition}
\begin{proof}See Corollary 2.3 of \cite{schenzel1999dimension}. The proof of the general case can be carried from that of the local case.
\end{proof}
\begin{corollary}For any $M$, the following holds:
\item (1) $M/M_i$ has no nonzero submodule of dimension at most $i$.

\item (2) $M_i=0$ if and only if for any $\mathfrak{p} \in \Ass(M)$, $\dim(A/\mathfrak{p}) > i$. If $A$ is a local catenary domain, this is equivalent to $\textup{ht}(\mathfrak{p}) < \dim(A) - i$.
\end{corollary}
We will take $A=R$ in the following part of the paper. For a general module $M$, the filtration $0 \subset M_0 \subset M_1 \subset M$ induces a decomposition of local cohomology tables.
\begin{lemma}\label{3.3}We have $M_0=H^0_{\mathfrak{m}}(M)$, and the exact sequence $0 \to M_0 \to M \to M/M_0 \to 0$ induces a decomposition of local cohomology tables, and $\depth(M/M_0) \geq 1$.
\end{lemma}
The proof is trivial and we omit it.
\begin{lemma}\label{3.4}Assume $\depth(M) \geq 1$. Then the exact sequence $0 \to M_1 \to M \to M/M_1 \to 0$ induces a decomposition of local cohomology tables.
\end{lemma}
\begin{proof}
We have depth$M_1 \geq$ 1 because $M_1$ is a submodule of $M$ and depth$M \geq$ 1. This means that $M_1$ is either 0 or Cohen-Macaulay of dimension 1. Also, $M/M_1$ does not have submodule of dimension at most 1; hence $H^0_{\mathfrak{m}}(M/M_1)=0$. Now the long exact sequence of local cohomology modules breaks up into short exact sequences:
$$0 \to H^1_{\mathfrak{m}}(M_1) \to H^1_{\mathfrak{m}}(M) \to H^1_{\mathfrak{m}}(M/M_1)  \to 0$$
and
$$0  \to H^i_{\mathfrak{m}}(M) \to H^i_{\mathfrak{m}}(M/M_1) \to 0$$
for any $i \geq 2$, and all the connecting homomorphisms are 0, so $H^{\bullet}(M)=H^{\bullet}(M_1)+H^{\bullet}(M/M_1)$.
\end{proof}
The dimension filtration relates a module $M$ to its dimension factors $M_i/M_{i-1}$. Sometimes, this relation may give us a decomposition of $H^\bullet(M)$ into $H^\bullet(M_i/M_{i-1})$. In \cite{CAVIGLIA2021106635}, Caviglia and De Stefani studied a class of modules called sequentially Cohen-Macaulay modules. This concept is first defined by Stanley \cite{Stanley1983CombinatoricsAC}. By definition, a module is sequentially Cohen-Macaulay if all its nonzero dimension factors are Cohen-Macaulay. The decomposition of local cohomology table of a sequentially Cohen-Macaulay module is given by the following result:
\begin{theorem}[Caviglia-De Stefani \cite{CAVIGLIA2021106635}]\label{theorem-decomposition of SCM module}
Let $M$ be a sequentially Cohen-Macaulay module and $M_i$ be its largest submodule of dimension at most $i$, then
$$H^\bullet(M)=\sum_{0 \leq i \leq \dim(M)} H^\bullet(M_i/M_{i-1}).$$
\end{theorem}
In this case, using the fact that the dimension factors are Cohen-Macaulay, we can decompose the tables $H^\bullet(M_i/M_{i-1})$ further into tables of the form $$H^\bullet(R/(x_{e+1},\ldots,x_n)[d])$$ for some number $e$ of variables and shifts $d \in \mathbb{Z}$, so the decomposition is clear.

Recall that a module $M$ is almost Cohen-Macaulay if $\depth(M)=\dim(M)-1$. We give a more general definition called sequentially almost Cohen-Macaulay:
\begin{definition}\label{definition-sacm}
Let $M$ be a finitely generated $A$-module. We say $M$ is sequentially almost Cohen-Macaulay (saCM for short) if all its nonzero dimension factors are Cohen-Macaulay or almost Cohen-Macaulay.
\end{definition}
The $i$-th dimension factor $M_i/M_{i-1}$ of $M$ must have dimension $i$ if it is nonzero, so if $M$ is sequentially Cohen-Macaulay then $\depth(M_i/M_{i-1})=i$, and if $M$ is saCM then $\depth(M_i/M_{i-1})=i$ or $i-1$. Note that by definition sequentially Cohen-Macaulay implies saCM.

Theorem \ref{theorem-decomposition of SCM module} for sequentially Cohen-Macaulay is still true for saCM modules: the local cohomology table of an saCM module decomposes into local cohomology tables of its dimension factors. More precisely, we have:
\begin{proposition}\label{3.6}Let $M$ be an saCM module, $M_i$ be its largest submodule of dimension at most $i$.
\item (1) $\depth(M/M_i) \geq i$ for any $0 \leq i \leq n-1$.
\item (2) $0 \to M_i/M_{i-1} \to M/M_{i-1} \to M/M_i \to 0$ induces a decomposition in local cohomology tables.
\item (3) $H^{\bullet}(M)=\sum^n_{i=0}H^{\bullet}(M_i/M_{i-1})$.
\end{proposition}
\begin{proof}
\item (1) We prove this by induction from $i=n-1$; in this case $\depth(M/M_{n-1}) \geq n-1$ by the saCM assumption because $M/M_{n-1}$ is just the $n$-th dimension factor. Suppose (1) is true for $i$. Consider the exact sequence $0 \to M_i/M_{i-1} \to M/M_{i-1} \to M/M_i \to 0$. We have $\depth(M_i/M_{i-1}) \geq i-1$ by the saCM assumption and $\depth(M/M_i) \geq i$ by the induction hypothesis, so by the depth lemma, $\depth(M/M_{i-1}) \geq i-1$, which implies that (1) is true for $i-1$, hence by induction (1) is true for any $0 \leq i \leq n-1$.
\item (2) The boundary maps of the long exact sequence of local cohomology modules are $H^k_{\mathfrak{m}}(M/M_i) \to H^{k+1}_{\mathfrak{m}}(M_i/M_{i-1})$. We have $\dim(M_i)/M_{i-1} \leq i$ and $\depth(M/M_i) \geq i$. Thus if $k \leq i-1$ then the source of the map is 0; if $k \geq i$ then the target is 0. Therefore, the boundary maps are 0 and the exact sequence induces a decomposition in local cohomology tables.
\item (3) Apply (2) inductively.

\end{proof}
It follows that to find the decomposition of local cohomology tables of saCM modules, we only need to decompose the tables of almost Cohen-Macaulay modules.

There is an important principle for modules of lower dimension: if $\dim(M) < \dim(R)$, then to find the decomposition of $H^{\bullet}(M)$, we may always replace $R$ with a new polynomial ring $S$ with $\dim(S)=\dim(M)$. This is done by the following lemma:
\begin{lemma}\label{3.7}
Let $R=k[x_1,\ldots,x_n]$ and $S=k[y_1,\ldots,y_d]$ be two standard graded polynomial rings with $n \geq d$. For a finitely generated graded $R$-module $M$ with $\dim(M) \leq d$, there is a finitely generated graded $S$-module $N$ such that $H^{\bullet}(M)$ can be obtained from $H^{\bullet}(N)$ by multiplying a positive rational scalar and adding 0's.
\end{lemma}
\begin{proof}If $k$ is infinite, see Lemma 2.2 of \cite{de2021decomposition}. If $k$ is finite, let $l=\bar{k}$ be its algebraic closure. Then by Lemma 2.2 of \cite{de2021decomposition}, there is an $S_l=l[y_1,\ldots,y_d]$-module $N$ such that $H^{\bullet}(N)$ and $H^{\bullet}(M)$ differ by a positive rational scalar. But $N$ is finitely generated, so there is a finite extension $k'$ of $k$ in $l$ such that $N$ is extended from $S_{k'}=k'[y_1,\ldots,y_d]$, that is, there is an $S_{k'}$-module $N_{k'}$ such that $N=N_{k'}\otimes_{k'}l$. Then $H^{\bullet}(N)=H^{\bullet}(N_{k'})$. Since $S=k[y_1,\ldots,y_d]$ embeds into $S_{k'}$ and this is module-finite, we can endow $N_{k'}$ with an $S$-module structure and it becomes a finitely generated graded $S$-module, say $M'$. Then $H^{\bullet}(N_{k'})=H^{\bullet}(M')[k':k]$. Thus, $H^{\bullet}(M)$ and $H^{\bullet}(M')$ only differ by a positive rational scalar.
\end{proof}
\begin{remark}Lemma \ref{3.7} is a variant of Lemma 2.2 of \cite{de2021decomposition} except that we do not change the base field. The table $H^{\bullet}(M)$ consists of $n+1$ series and the table $H^{\bullet}(N)$ consists of $d+1$ series so we need to add $n-d$ 0's to make their sizes equal.
\end{remark}
The 0 entries in $H^{\bullet}(M)$ remains 0 in $H^{\bullet}(N)$, so for any choice of $N$, we have $\depth(M)=\depth(N)$ and $\dim(M)=\dim(N)$. If we fix a choice of $N$ for every $M$ and view it as a correspondence, then almost Cohen-Macaulay modules correspond to almost Cohen-Macaulay modules of maximal dimension, which are just the modules with maximal dimension and projective dimension 1. We will find how to decompose the local cohomology tables of modules of projective dimension 1 in Section \ref{section-pd<=1case}.
\begin{remark}Let $R=k[x_1,\ldots,x_n]$ and $S=k[y_1,\ldots,y_d]$ be two standard graded polynomial rings with $n \geq d$. Then $S \cong R/(x_n,\ldots,x_{d+1})$. The local cohomology table of an $S$-module $M$ does not change if we view $M$ as an $R$-module, so there is a natural inclusion from $\{H^{\bullet}(M), M$ is an $S$-module$\}$ to $\{H^{\bullet}(M), M$ is an $R$-module$\}$. Hence there is a natural inclusion of their cones, that is, $\mathbb{Q}_{\geq 0}\{H^{\bullet}(M), M$ is an $S$-module$\}$ injects into $\mathbb{Q}_{\geq 0}\{H^{\bullet}(M), M$ is an $R$-module$\}$. By Lemma \ref{3.7}, the image is just $\mathbb{Q}_{\geq 0}\{H^{\bullet}(M), M$ is an $R$-module, $\dim(M) \leq \dim(S)\}$. In this sense, when we study the decomposition of local cohomology tables of $R$-modules of dimension at most $d$, it suffices to study the decomposition of local cohomology tables of $S$-modules.
\end{remark}

\section{Projective dimension 1 case}\label{section-pd<=1case}
This section describes the cone generated by $H^{\bullet}(M)$ where $M$ is an $R$-module with projdim($M$) $\leq$ 1, how the tables in this cone decompose, and how the decomposition of such tables leads to the decomposition of tables of saCM modules.

First, let us recall the definition of Auslander transpose introduced by Auslander and Bridger in \cite{auslander1969stable}. Let $\cdot^*=\Hom_R(\cdot,R)$ be the dual functor.
\begin{definition}\label{Sec4-auslander transpose}
Let $M$ be a finitely generated module over a $R$. Consider a finite presentation $F_1 \xrightarrow{\phi} F_0 \to M \to 0$. Taking dual yields an exact sequence $0 \to M^* \to F_0^* \xrightarrow{\phi^*} F_1^* \to N \to 0$. Then the \emph{Auslander transpose} of $M$ is $\Tr(M) = N = \textup{Coker}(\phi^*)$. If $\phi$ is a minimal presentation, then $\Tr(M)$ is called the \emph{minimal Auslander transpose}.
\end{definition}
\begin{remark}In general, the Auslander transpose is unique up to a projective summand. However, the minimal Auslander transpose is unique up to isomorphism; so in the following sections, when we mention the module $\Tr(M)$, we always mean the minimal Auslander transpose.
\end{remark}
Here are some basic properties of the Auslander transpose.
\begin{proposition}
\item (1) The Auslander transpose of a graded module is also graded.

\item (2) If \textup{projdim}$(M) \leq 1$, then $\Tr(M) = \Ext^1_R(M,R)$.

\item (3) $\Tr(M)=0$ if and only if \textup{projdim}$(M)=0$, that is, $M$ is free.

\item (4) $\Tr(M \oplus M')=\Tr(M) \oplus \Tr(M')$.

\item (5) $M=\Tr(\Tr(M))\oplus F$, $F$ is free and $\Tr(\Tr(M))$ does not have a free summand.
\end{proposition}
Let $M$ be a finitely generated graded $R$-module. To study the decomposition of $H^{\bullet}(M)$, we may assume $M$ is indecomposable without loss of generality. If projdim$M=0$ then $M$ is free, and we must have $M \cong R(-i)$ for some $i$. In this case $H^{\bullet}(M)$ is clear. So we may assume that projdim$(M) = 1$. Here is an important observation about the properties of $M$ and $\Tr(M)$.
\begin{proposition}\label{4.4}
\item (1) Let $M$ be a finitely generated $R$-module of projective dimension 1 with no free summand. Let $0 \to F_1 \xrightarrow{\phi} F_0 \to M \to 0$ be a minimal presentation which is also a minimal resolution. Let $N=\Tr(M) \neq 0$, then \textup{dim}$N$ $\leq$ $n-1$, and $F_0^* \xrightarrow{\phi^*} F_1^* \to N \to 0$ is a minimal presentation of $N$.
\item (2) Let $N$ be a nonzero module with \textup{dim}$(N)$ $\leq$ $n-1$. Take a minimal presentation $G_1 \xrightarrow{\psi} G_0 \to N \to 0$. Then $G_0^* \xrightarrow{\psi^*} G_1^*$ is injective and its image lies in $\mathfrak{m}G_1^*$, so if $M=\Tr(N)=\textup{Coker}(\psi^*)$, then \textup{projdim}$M = 1$ and $M$ has a minimal resolution $0 \to G_0^* \xrightarrow{\psi^*} G_1^* \to M \to 0$. Also, $M$ does not have a free summand.
\item (3) Taking the minimal Auslander transpose $\Tr$ induces a 1-1 correspondence between the isomorphism classes of finitely generated graded $R$-modules $M$ of projective dimension 1 without a free summand and finitely generated graded $R$-modules $N$ of dimension at most $n-1$.
\item (4) Under the assumption in (3), $\beta_{0,j}(M)=\beta_{1,-j}(N)$, $\beta_{1,j}(M)=\beta_{0,-j}(N)$, or equivalently, $\beta_0(M)(t)=\beta_1(N)(t^{-1}), \beta_1(M)(t)=\beta_0(N)(t^{-1})$.
\end{proposition}
\begin{proof}
\item (1) Suppose projdim$M$ = 1 with minimal resolution $0 \to F_1 \xrightarrow{\phi} F_0 \to M \to 0$. Then $N =$ Ext$^1_R(M,R)$, so dim$N$ $\leq$ $n-1$. Now consider the exact sequence $F_0^* \xrightarrow{\phi^*} F_1^* \to N \to 0$. Since $\phi$ has entries in $\mathfrak{m}$, so does $\phi^*$, so any free basis of $F_1^*$ maps to a minimal generating set of $N$. Now $F_0^*$ surjects onto Syz$_1(N)$; if the image of one basis element of $F_0^*$ is not a minimal generating set of the image of $\phi^*$, then some basis element $e_i$ of $F_0^*$ is mapped to $0$ under $\phi^*$. Then taking the dual again, $Re_i^*$ will become a free summand of $M$, contradicting with our assumption. So the image of the basis of $F_0^*$ is a minimal generating set of Syz$_1(N)$. This means that $F_0^* \to F_1^* \to N \to 0$ is a minimal presentation of $N$.

\item (2) Take a minimal presentation $G_1 \xrightarrow{\psi} G_0 \to N \to 0$. Let $M=\Tr(N)$, then $G_0^* \xrightarrow{\psi^*} G_1^* \to M \to 0$ is exact. By minimality $\psi$ has entries in $\mathfrak{m}$, hence so does $\psi^*$. Let $K$ = Quot($R$) the quotient field of $R$. Then $N \otimes K=0$, hence $\psi \otimes K$ is surjective and it is also a $K$-linear map where $K$ is a field. So $(\psi \otimes K)^*=\psi^* \otimes K$ is injective. This implies that Ker$(\psi^*)$ is torsion, but it is a submodule of $G_0^*$, so it must be 0. In other words, $\psi^*$ is injective and this means $0 \to G_0^* \to G_1^* \to M \to 0$ is a minimal free resolution of $M$. If $M$ has a free summand, it must be generated by the image of some basis elements of $G_1^*$. Pick one of these basis elements $e_i$ and expand it to a basis of $G_1^*$, then we know that the $e_i$-coefficient of all elements in $\psi(G_0^*)$ is 0. Taking dual again, we get $\psi(e_i^*)=0$, which means that $G_1$ is not mapped to Syz$_1(N)$ minimally. This is a contradiction. Hence $M$ has no free summands.
\item (3) Obvious by (1) and (2).
\item (4) Obvious by (1), (2) and (3).
\end{proof}
The above proposition reveals that the following assumption is essential in the analysis of Betti numbers. We will make this assumption frequently in the following propositions.
\begin{assumption}\label{assumptions-Mpd1,NtrM}
Assume $M$ is a finitely generated graded $R$-module without a free summand with $\textup{projdim}(M)=1$, and let $N=\Tr(M)$.    
\end{assumption}
The Betti table of the Auslander transpose of $M$ describes the local cohomology table of $M$ under Assumption \ref{assumptions-Mpd1,NtrM}.
\begin{proposition}\label{4.5}
Assume $M,N$ satisfies Assumption \ref{assumptions-Mpd1,NtrM}, then $E^{\bullet}(M)=(1-t)^{-n}(\sum_{i=2}^{n}(-1)^i\beta_i(N)$, $\sum_{i=0}^{n}(-1)^i\beta_i(N),0,\ldots,0)$.
\end{proposition}
\begin{proof}Let $0 \to F_1 \to F_0 \to M \to 0$ be the minimal resolution of $M$. This induces an exact sequence $0 \to M^* \to F_0^* \to F_1^* \to N \to 0$ where $F_0^* \to F_1^* \to N \to 0$ is a minimal representation. By definition, $E^{\bullet}(M)=(HS(M^*),HS(N),0,\ldots,0)$. Now $F_0^* \to F_1^* \to N \to 0$ is a minimal presentation, so $HS(F_1^*)=\beta_0(N)HS(R)=\beta_0(N)(1-t)^{-n}, HS(F_0^*)=\beta_1(N)HS(R)=\beta_1(N)(1-t)^{-n}$. Now $HS(N)=(\sum_{i=0}^{n}(-1)^i\beta_i(N))(1-t)^{-n}$ and by the long exact sequence, $HS(M^*)=HS(F_0^*)+HS(N)-HS(F_1^*)=(\sum_{i=2}^{n}(-1)^i\beta_i(N))(1-t)^{-n}$.
\end{proof}
\begin{corollary}\label{4.6}
Let $V$ be the space of Betti tables and Ext tables, $M$, $N$ be two modules satisfying Assumption \ref{assumptions-Mpd1,NtrM}. Define $L_1:V \to V$ to be the $\mathbb{Q}$-linear map $(\beta_0,\beta_1,\ldots,\beta_n) \to (1-t)^{-n}(\sum_{i=2}^{n}(-1)^i\beta_i,\sum_{i=0}^{n}(-1)^i\beta_i,0,\ldots,0)$, then $E^{\bullet}(M)=L_1(\beta^{\bullet}(N))$.
\end{corollary}
\begin{corollary}\label{4.7}
Let $C_{wf}$ be the cone in $V$ generated by the Ext-tables of modules of projective dimension $1$ which does not have a free summand. Then if $E^{\bullet}(M)$ is an extremal ray and $N=\Tr(M)$, then $N$ has a pure resolution of length at least $1$, and every element in $C_{wf}$ is a positive linear combination of elements of the form $E^{\bullet}(M)$, where $\Tr(M)$ has a pure resolution of length at least $1$.
\end{corollary}
\begin{proof}Let $C_b$ be the cone generated by all Betti tables of modules of dimension at most $n-1$. Then by Proposition \ref{4.4} (3) and Corollary \ref{4.6}, $C_{wf}=L_1(C_b)$. Applying the Boij-S\"oderberg theory for Betti tables we know that the extremal rays of $C_b$ are the Betti tables of modules with pure resolutions of length $s$, where $1 \leq s \leq n$ and $C_b$ is generated by these elements as a cone. Now apply Lemma \ref{2.1}.
\end{proof}
By the proposition above, for $M,N$ satisfying Assumption \ref{assumptions-Mpd1,NtrM}, we already know how to decompose $E^{\bullet}(M)$ when $N$ does not have a pure resolution, so to find the vertices of the cone of Ext-tables it suffices to analyze when $L_1(\beta^{\bullet}(N))$ is decomposable, where $N$ has a pure resolution. By Corollary \ref{4.7} we may always assume the length of the resolution $s \geq 1$, and $\dim(N) \leq n-1$. Now let $N$ be a finitely generated graded module of dimension at most $n-1$. We will say $N$ is pure of type $\textbf{d}$ for a degree sequence $\textbf{d}$ if the minimal free resolution of $N$ is pure of type $\textbf{d}$. First, we need two lemmas that allow us to compute $\beta^{\bullet}(N)$ when $N$ is pure in terms of its degree sequence $\textbf{d}=(d_0,\ldots,d_s)$, and in this case $s \geq 1$. 
\begin{lemma}Let $s_1 \leq s$ be two nonnegative integers, $\textbf{d}=(d_0,\ldots,d_s)$ be a degree sequence, and $V_{\textbf{d},s_1}$ be the vector space $\{f \in \mathbb{Q}[t,t^{-1}] \textup{such that} f=\sum_{i=0}^{s}\pi_{\textbf{d},d_i}t^{d_i},(1-t)^{s_1}|f\}$. Then $\dim_{\mathbb{Q}}V_{\textbf{d},s_1}=s-s_1+1$.
\end{lemma}
\begin{proof}Multiplying by $t^{-d_0}$ does not affect the order of the pole at $t=1$, so we may assume $d_0=0$ without loss of generality. In this case, every element in $V_{\textbf{d},s_1}$ will be a polynomial. For a polynomial $f$, $(1-t)^{s_1}$ divides $f$ if and only if
$$d^jf/dt^j(1)=0, \forall j=0,1,\ldots,s_1-1.$$
Assume $f=\sum_{i=0}^{s}\pi_{\textbf{d},d_i}t^{d_i}$, then we have
$$\sum_{i=0}^{s}\pi_{\textbf{d},d_i} {{d_i}\choose{j}} = 0, \forall j=0,1,\ldots,s_1-1.$$
We can write the above equation as $Ax=0$ where $A=({{d_i}\choose{j}})_{0 \leq i \leq s, 0 \leq j \leq s_1-1}$ is a matrix and $x$ is a column vector. Let $B=(d^j_i)_{0 \leq i \leq s, 0 \leq j \leq s_1-1}$ which is a matrix of the same size as $A$. It is easy to see that we can apply elementary row operation to $A$ and get $B$ since we can expand ${{d_i}\choose{j}}$ into a polynomial of degree $j$ and clear the lower degree terms of each row using the other rows, thus $Ax=0$ and $Bx=0$ has the same solution space. The matrix $B$ has full rank $s_1$ because it has a Vandermonde submatrix of rank $s_1$ and we have $s+1$ variables, so the dimension of the solution space is $s-s_1+1$.
\end{proof}
In the case where $s=s_1$, dim$_{\mathbb{Q}}V_{\textbf{d},s} = 1$, so there is a unique vector up to a scalar. Denote the sign function by $sgn$, then this vector has the alternating sign property, described as below.
\begin{lemma}\label{4.9}
\item (1) For each degree sequence $\textbf{d}=(d_0,\ldots,d_s)$, \textup{dim}$_{\mathbb{Q}}V_{\textbf{d},s}=1$, hence there exists a unique polynomial $\pi_{\textbf{d}}(t) \in \mathbb{Q}[t,t^{-1}]$ up to multiplying by a nonzero rational number inside $V_{\textbf{d},s}$, denoted by $\pi_{\textbf{d}}(t)=\sum_{i=0}^{s}\pi_{\textbf{d},d_i}t^{d_i}$.
\item (2) If we rescale these coefficients so that $\pi_{\textbf{d},d_0}=1$, then $\pi_{\textbf{d},d_i}=\frac{\Pi_{j \neq 0}(d_j-d_0)}{\Pi_{j \neq i}(d_j-d_i)}$, and $sgn(\pi_{\textbf{d},d_i})=(-1)^i$, that is, the coefficients are nonzero and have alternating signs.
\item (3) Under the assumption in (2), $\pi_{\textbf{d}}(t)/(1-t)^s|_{t=1} > 0$.
\item (4) Up to multiplying by a scalar, $\beta^{\bullet}(N)=((-1)^i\pi_{\textbf{d},d_i}t^{d_i})$ if $N$ is pure of type $\textbf{d}$.
\end{lemma}
The proof is shown in the beginning of Section 2.1 of \cite{boij2008graded} up to Definition 2.3. For example, we have $\pi_{0,1,2,3}(t)=1-3t+3t^2-t^3$, and $\pi_{0,1,3,4}(t)=1-2t+2t^3-t^4$.

We need to introduce more notions to analyze when $L_1(\beta^\bullet(N))$, where $N$ is pure, can decompose nontrivially. We define another invertible linear map
$L_2:V \to V$, $L_2(f_0,f_1,f_2,\ldots,f_n)=(1-t)^n(f_1-f_0,f_0,f_2,\ldots,f_n)$.
Since $L_2$ is an isomorphism of $\mathbb{Q}$-vector spaces, it induces a bijection between the vertex set of a cone with the vertex set of the image of the cone by Lemma \ref{2.1}. Under this notation, $L_2L_1(\beta^{\bullet}(N))=(\beta_0(N)-\beta_1(N),\sum_{i=2}^{n}(-1)^i\beta_i(N),0,\ldots,0)$. In other words, the element $L_2L_1(\beta^{\bullet}(N))$ has two nonzero entries; the first entry is the sum of two terms of lowest degree in $\pi_{\textbf{d}}(t)$ and the second entry is the sum of the $n-1$ terms of highest degree in $\pi_{\textbf{d}}(t)$. Hence it is natural to introduce the following notation for a degree sequence $\textbf{d}$: let $\alpha_{\textbf{d}}(t)=\pi_{\textbf{d},d_0}t^{d_0}+\pi_{\textbf{d},d_1}t^{d_1}$ and $\alpha_{\textbf{d}}'(t)=\sum_{i=2}^{s}\pi_{\textbf{d},d_i}t^{d_i}=\pi_{\textbf{d}}(t)-\alpha_{\textbf{d}}(t)$. More generally, for an integer $d$, define $\tau_{\leq d}$ to be a map that sends a Laurent polynomial $f \in \mathbb{Q}[t,t^{-1}]$ to the sum of terms of $f$ of degree less than $d$ and $\tau_{\geq d}$, that sends a Laurent polynomial $f \in \mathbb{Q}[t,t^{-1}]$ to the sum of terms of $f$ of degree at least $d$. It is easy to see that for a degree sequence $\textbf{d}=(d_0<d_1<\ldots<d_s)$, $\alpha_{\textbf{d}}(t)=\tau_{\leq d_1}\pi_{\textbf{d}}(t)$ and $\alpha_{\textbf{d}}'(t)=\tau_{\geq d_2}\pi_{\textbf{d}}(t)$. We have:
\begin{proposition}\label{4.10} 
Assume $N$ is pure of type $\textbf{d}$. Then $L_2L_1(\beta^{\bullet}(N))=(\alpha_{\textbf{d}}(t),\alpha_{\textbf{d}}'(t))$.
\end{proposition}
We need to check whether $L_2L_1(\beta^{\bullet}(N))$ can be decomposed for various $\textbf{d}$'s. The next proposition shows that if there is a space between $d_i$ and $d_{i+1}$ for $i \neq 1$, then we can decompose $L_2L_1(\beta^{\bullet}(N))$.
\begin{proposition}\label{4.11}
Let $s \geq 1$ be a positive integer. Let $\textbf{d}=(d_0 < d_1 <\ldots< d_s)$ be a degree sequence. Assume $d_{i}<d_{i+1}-1$ for some $i$ and pick an integer $a$ such that $d_i < a < d_{i+1}$. Define two degree sequences $\textbf{d'}=(d_0 < d_1 <\ldots<d_i <a <d_{i+2}< d_s)$ and $\textbf{d''}=(d_0 < d_1 <\ldots<d_{i-1} <a <d_{i+1}< d_s)$. Then $\pi_{\textbf{d}}(t)=c_1\pi_{\textbf{d'}}(t)+c_2\pi_{\textbf{d''}}(t)$ where $c_1,c_2>0$. Moreover, if $i \neq 1$, then we also have $\alpha_{\textbf{d}}(t)=c_1\alpha_{\textbf{d'}}(t)+c_2\alpha_{\textbf{d''}}(t)$ and $\alpha_{\textbf{d}}'(t)=c_1\alpha_{\textbf{d'}}'(t)+c_2\alpha_{\textbf{d''}}'(t)$, so in particular, let $N$ be a pure module of type $\textbf{d}$, then $L_2L_1(\beta^{\bullet}(N))=c'_1L_2L_1(\beta^{\bullet}(N'))+c'_2L_2L_1(\beta^{\bullet}(N''))$ where $N'$ is pure of type $\textbf{d'}$, $N''$ is pure of type $\textbf{d''}$, and $c'_1, c'_2>0$ are elements in $\mathbb{Q}$.
\end{proposition}
\begin{proof}If $s=1$, then  $\pi_{\textbf{d}}(t)=t^{d_0}-t^{d_1}$, $\pi_{\textbf{d'}}(t)=t^{d_0}-t^a$, $\pi_{\textbf{d''}}(t)=t^a-t^{d_1}$, so $\pi_{\textbf{d}}(t)=\pi_{\textbf{d'}}(t)+\pi_{\textbf{d''}}(t)$. Now assume $s \geq 2$, then by Lemma \ref{4.9} (2) we know $\pi_{\textbf{d'}}(t)$ has nonzero coefficients at degree $d_j, j \neq i+1$ and degree $a$, and $\pi_{\textbf{d''}}(t)$ has nonzero coefficients at degree $d_j, j \neq i$ and degree $a$. So by cancelling the coefficients in degree $a$, there is a linear combination $c_1\pi_{\textbf{d'}}(t)+c_2\pi_{\textbf{d''}}(t)$ which is a polynomial with possible nonzero coefficients at degree $d_i, 0 \leq i \leq s$. Now $sgn(\pi_{\textbf{d'},a})=(-1)^{i+1} \neq sgn(\pi_{\textbf{d''},a})=(-1)^i$. Hence we have $sgn(c_1)=sgn(c_2)$. This polynomial is still divisible by $(1-t)^s$, so by Lemma \ref{4.9} (1), it is a multiple of $\pi_{\textbf{d}}(t)$, and after rescaling we may assume $\pi_{\textbf{d}}(t)=c_1\pi_{\textbf{d'}}(t)+c_2\pi_{\textbf{d''}}(t)$ and $sgn(c_1)=sgn(c_2)$. Now since $s \geq 2$, we have $i \geq 1$ or $i+1 \leq s-1$. In the first case, $sgn(\pi_{\textbf{d},d_0})=sgn(\pi_{\textbf{d'},d_0})=sgn(\pi_{\textbf{d''},d_0})=1$ and in the second case $sgn(\pi_{\textbf{d},d_s})=sgn(\pi_{\textbf{d'},d_s})=sgn(\pi_{\textbf{d''},d_s})=(-1)^s$, so $c_1,c_2 >0$. If $i \neq 1$, then either $i=0,i+1=1$ or $i \geq 2$. We can apply $\tau_{\leq d_1}$ to the equation $\pi_{\textbf{d}}(t)=c_1\pi_{\textbf{d'}}(t)+c_2\pi_{\textbf{d''}}(t)$ to get $\alpha_{\textbf{d}}(t)=c_1\alpha_{\textbf{d'}}(t)+c_2\alpha_{\textbf{d''}}(t)$ and apply $\tau_{\geq d_2}$ to get $\alpha_{\textbf{d}}'(t)=c_1\alpha_{\textbf{d'}}'(t)+c_2\alpha_{\textbf{d''}}'(t)$. The last statement is true for $c'_1=c_1$ and $c'_2=c_2$ by Proposition 4.10.
\end{proof}
\begin{definition}\label{property-P}
Let $\textbf{d}=(d_{0} < d_{1} <\ldots< d_{s})$ be a degree sequence. We say $\textbf{d}$ satisfies condition $\mathcal{P}$ if $d_1-d_0=1$ and $d_j=d_2+j-2$ for $2 \leq j \leq s$.      
\end{definition}
\begin{corollary}\label{4.12}
Let $\textbf{d}_0=(d_{0,0} < d_{0,1} <\ldots< d_{0,s})$ be a degree sequence. and $N$ be pure of type $\textbf{d}_0$, then there exist a collection of $N_i$'s which are pure of type $\textbf{d}_i$'s and those $\textbf{d}_i$'s satisfy $\mathcal{P}$, such that $L_2L_1(\beta^{\bullet}(N))$ decomposes into $L_2L_1(\beta^{\bullet}(N_i))$.
\end{corollary}
\begin{proof}We fix the degree sequence $\textbf{d}_0=(d_{0,0} < d_{0,1} <\ldots< d_{0,s})$. Let $A$ be the set of degree sequences $\{\textbf{d}=(d_0 < d_1 <\ldots< d_s) | d_{0,0} \leq d_0 \leq d_1 \leq d_{0,1}, d_{0,2} \leq d_2 \leq d_s \leq d_{0,s}\}$. Then $A$ is a finite set since $d_{0,0}, d_{0,1}, d_{0,2}, d_{0,s}$ are fixed. If $N$ is pure of type $\textbf{d}$ that does not satisfy $\mathcal{P}$, then $\textbf{d}$ satisfies the hypothesis of Proposition \ref{4.10} so $L_2L_1(\beta^{\bullet}(N))$ decomposes, and moreover, using the notation in Proposition \ref{4.11}, the two degree sequences $\textbf{d}',\textbf{d}''$ are still in $A$. Let $C'$ be the cone generated by $L_2L_1(\beta^{\bullet}(N))$, where $N$ is pure of type $\textbf{d} \in A$. Consider the set $B=\{L_2L_1(\beta^{\bullet}(N))| N$ pure of type $\textbf{d}$, $\textbf{d} \in A$ satisfies $\mathcal{P}\}$. Then $B$ contains the vertex set because every element in $C' \backslash B$ can be decomposed into elements in $C'$. We know $C'$ is finitely generated as a cone, so a vertex set of $C'$ also generates $C'$ by Theorem 1.26 of \cite{bruns2009polytopes}, hence $B$ generates $C'$, and $L_2L_1(\beta^{\bullet}(N)) \in C$. This means that $L_2L_1(\beta^{\bullet}(N))$ decomposes into elements in $B$, which proves the corollary.
\end{proof}
The next proposition shows that $L_2L_1(\beta^{\bullet}(N))$ is decomposable if $s \neq 1,n$. Note that we always assume $1 \leq s \leq n$.
\begin{proposition}\label{4.13}Let $2 \leq s \leq n-1$. Let $\textbf{d}=(d_0 < d_1 <\ldots< d_s)$ be a degree sequence. Construct two degree sequences $\textbf{d'}=(d_0 < d_1 <\ldots< d_{s-1}<d_s+1)$ and $\textbf{d''}=(d_0 < d_1 <\ldots< d_{s-1}<d_s<d_s+1)$. The first degree sequence has length $s+1$ and the second degree sequence has length $s+2$. Then $c_1\pi_{\textbf{d}}(t)+c_2\pi_{\textbf{d'}}(t)=\pi_{\textbf{d''}}(t)$ where $c_1>0$ and $c_2<0$. Moreover we also have $\alpha_{\textbf{d''}}(t)=c_1\alpha_{\textbf{d}}(t)+c_2\alpha_{\textbf{d'}}(t)$ and $\alpha_{\textbf{d''}}'(t)=c_1\alpha_{\textbf{d}}'(t)+c_2\alpha_{\textbf{d'}}'(t)$. In particular, let $N$ be a pure module of type $\textbf{d}$, then $L_2L_1(\beta^{\bullet}(N))=c'_1L_2L_1(\beta^{\bullet}(N'))+c'_2L_2L_1(\beta^{\bullet}(N''))$ where $N'$ is pure of type $\textbf{d'}$ and $N''$ is pure of type $\textbf{d''}$, where $c'_1, c'_2>0$ are elements in $\mathbb{Q}$.
\end{proposition}
\begin{proof}Consider the $\mathbb{Q}$-vector space spanned by $\pi_{\textbf{d}}(t)$ and $\pi_{\textbf{d'}}(t)$. The two polynomials are linearly independent because $\pi_{\textbf{d},d_s} \neq 0, \pi_{\textbf{d'},d_s} =0, \pi_{\textbf{d},d_s+1}=0, \pi_{\textbf{d},d_s} \neq 0$. So they span $V_{\textbf{d''},s}$. Also $(1-t)^{s+1}|\pi_{\textbf{d''}}(t)$, hence there exist $c_1,c_2 \in \mathbb{Q}$ such that $c_1\pi_{\textbf{d}}(t)+c_2\pi_{\textbf{d'}}(t)=\pi_{\textbf{d''}}(t)$. Now $sgn(c_1)=sgn(\pi_{\textbf{d},d_s})/sgn(\pi_{\textbf{d''},d_s})=(-1)^s/(-1)^s=1$ and $sgn(c_2)=sgn(\pi_{\textbf{d'},d_s+1})/sgn(\pi_{\textbf{d''},d_s+1})=(-1)^s/(-1)^{s+1}=-1$. Since $s \neq 1$, $s \geq 2$, by applying $\tau_{\leq d_1}$ and $\tau_{d_2}$ to this equation we get $\alpha_{\textbf{d''}}(t)=c_1\alpha_{\textbf{d}}(t)+c_2\alpha_{\textbf{d'}}(t)$ and $\alpha_{\textbf{d''}}'(t)=c_1\alpha_{\textbf{d}}'(t)+c_2\alpha_{\textbf{d'}}'(t)$. By Proposition \ref{4.10} this just means $L_2L_1(\beta^{\bullet}(N''))=c_1L_2L_1(\beta^{\bullet}(N))+c_2L_2L_1(\beta^{\bullet}(N''))$, therefore $$L_2L_1(\beta^{\bullet}(N))=-\frac{c_2}{c_1}L_2L_1(\beta^{\bullet}(N'))+\frac{1}{c_1}L_2L_1(\beta^{\bullet}(N'')).$$
Let $c'_1=-\frac{c_2}{c_1}$ and $c'_2=\frac{1}{c_1}$. Since $c_2 < 0 < c_1$, $c_1,c_2 \in \mathbb{Q}$, we know $c'_1, c'_2>0$ and $c'_1, c'_2 \in \mathbb{Q}$.
\end{proof}

Let $C_t$ be the cone generated by the Ext-tables of modules of projective dimension at most 1; let $C_{wf}$ be the same as Corollary \ref{4.7}, that is, the cone generated by the Ext-tables of modules of projective dimension 1 without a free summand; let $C_f$ be the cone generated by the Ext-tables of free modules. Then $C_t=C_{wf}+C_f$. We want to know a generating set and the vertex set of $C_t$. To simplify the expressions, we introduce some more notations.

For a module $N$ of type $\textbf{d}=d_0<d_1<\ldots<d_s$, denote $L_1(\beta^{\bullet}(N))=a_{\textbf{d}}$ and $L_2(a_{\textbf{d}})=b_{\textbf{d}}$.
If $N$ is pure of type $\textbf{d}=d_0<d_1<\ldots<d_s$ that does not satisfy the assumption in Proposition \ref{4.11} or Proposition \ref{4.13}, then $\textbf{d}$ must satisfy proposition $\mathcal{P}$ in Corollary \ref{property-P} and is of length 2 or $n+1$. So either $s=1, \textbf{d}=d_0<d_0+1$, or $s=n, \textbf{d}=d_0,d_0+1,d_2,d_2+1,\ldots,d_2+n-2$. Consider the following 4 kinds of tables that are Ext-tables of some modules:
\begin{enumerate}
\item $A_1=\{L_1(\beta^{\bullet}(N))=a_{\textbf{d}}= (1-t)^{-n}(0,t^{d_0}-t^{d_0+1},0,\ldots,0)$, where $N$ is pure of type $\textbf{d},$ $\textbf{d}=(d_0,d_0+1)\}$.

\item $A_2=\{L_1(\beta^{\bullet}(N))=a_{\textbf{d}}$, where $N$ is pure of type $\textbf{d},$ $\textbf{d}=(d_0,d_0+1,d_2,d_2+1,\ldots,d_2+n-2)\}$. In this case we have
    $$a_{\textbf{d}}= (1-t)^{-n}(\sum_{i=2}^{n}\pi_{\textbf{d},d_i}t^{d_i},\sum_{i=0}^{n}\pi_{\textbf{d},d_i}t^{d_i},0,\ldots,0).$$

\item $A_3=\{L_1(\beta^{\bullet}(N))=a_{\textbf{d}}$, where $N$ is pure of type $\textbf{d},$ $\textbf{d}=(d_0,d_0+1,d_2,d_2+1,\ldots,d_2+s-2), 2 \leq s \leq n-1\}$. In this case we have
    $$a_{\textbf{d}}= (1-t)^{-n}(\sum_{i=2}^{s}\pi_{\textbf{d},d_i}t^{d_i},\sum_{i=0}^{s}\pi_{\textbf{d},d_i}t^{d_i},0,\ldots,0).$$

\item $A_4=\{E^{\bullet}(R(d))=(1-t)^{-n}(t^d,0,\ldots,0), d \in \mathbb{Z}\}$.
\end{enumerate}
As a summary of the propositions above, we know:
\begin{proposition}\label{4.14}
\item (1) $C_{wf}$ is generated by $A_1 \cup A_2 \cup A_3$.

\item (2) $C_f$ is generated by $A_4$.

\item (3) $C_t$ is generated by $A_1 \cup A_2 \cup A_3 \cup A_4$.

\item (4) An element in $A_3$ decomposes into elements in $A_2 \cup A_3$, so it cannot be a vertex.
\end{proposition}
\begin{proof}
\item (1) This is proved by Corollary \ref{4.7} and Corollary \ref{4.12}.
\item (2) This is true because every free module is a direct sum of free modules of rank 1.
\item (3) It is trivial by (1) and (2).
\item (4) This is proved by Proposition \ref{4.13}.
\end{proof}
Therefore, to find the vertex set of $C_t$ it suffices to determine whether elements in $A_1 \cup A_2 \cup A_4$ decompose into elements in $A_1 \cup A_2 \cup A_3 \cup A_4$ nontrivially.
\begin{proposition}\label{4.15}The vertex set of $C_t$ is $A_1 \cup A_2 \cup A_4$.
\end{proposition}
\begin{proof}Observe that only elements in $A_1$ have a 0 entry in the first component and the elements in $A_2, A_3$ and $A_4$ have positive entries. So if an element in $A_1$ decomposes, it can only decompose into elements in $A_1$, but elements in $A_1$ are linearly independent, therefore, the decomposition is trivial. Similarly, checking the second component we know elements in $A_4$ only have trivial decompositions. So it remains to check elements in $A_2$. We apply $L_2$ again to the elements in $A_1, A_2, A_3$ and $A_4$. We have:
$$L_2((1-t)^{-n}(0,t^{d_0}-t^{d_0+1},0,\ldots,0))=(t^{d_0}-t^{d_0+1},0,0,\ldots,0),$$
$$L_2((1-t)^{-n}(\sum_{i=2}^{n}\pi_{\textbf{d},d_i}t^{d_i},\sum_{i=0}^{n}\pi_{\textbf{d},d_i}t^{d_i},0,\ldots,0))=$$
$$(\sum_{i=0}^{1}\pi_{\textbf{d},d_i}t^{d_i},\sum_{i=2}^{n}\pi_{\textbf{d},d_i}t^{d_i},0,\ldots,0),$$
$$L_2((1-t)^{-n}(\sum_{i=2}^{s}\pi_{\textbf{d},d_i}t^{d_i},\sum_{i=0}^{n}\pi_{\textbf{d},d_i}t^{d_i},0,\ldots,0))=$$
$$(\sum_{i=0}^{1}\pi_{\textbf{d},d_i}t^{d_i},\sum_{i=2}^{s}\pi_{\textbf{d},d_i}t^{d_i},0,\ldots,0),$$
$$L_2((1-t)^{-n}(t^d,0,\ldots,0))=(-t^d,t^d,0,\ldots,0).$$
The first 3 kinds of elements are also equal to $L_2(a_{\textbf{d}})=b_{\textbf{d}}$. Now assume we have an equation
\[b_{\textbf{d}_0}=\sum_{j \in J} q_jb_{\textbf{d}_j}+\sum_{k \in K} q_kb_{\textbf{d}_k}+\sum_{l \in L} q_lb_{\textbf{d}_l}+\sum_{m \in M} q_m(-t^m,t^m,0,\ldots,0). \tag{*}\]
with $b_{\textbf{d}_j},b_{\textbf{d}_k},b_{\textbf{d}_l}$ belonging to $L_2(A_1), L_2(A_2), L_2(A_3)$ respectively and $q_j, q_k, q_l, q_m$ being positive rational numbers. We prove that this decomposition is trivial in the following steps.

(1) Observe the following fact: for each $b_{\textbf{d}_j},b_{\textbf{d}_k},b_{\textbf{d}_l},(-t^m,t^m,0,\ldots,0)$, the lowest term of the second component has a positive coefficient. So let $b_{\min2}=\min\{\textbf{d}_{k,2}, \textbf{d}_{l,2}, m\}$, then $b_{\min2}=\textbf{d}_{0,2}$. In fact, if $b_{\min2}<\textbf{d}_{0,2}$ then on the right side of (*) the coefficient of $t^{b_{\min2}}$ in the second component is positive while on the left side it is 0. If $b_{\min2}>\textbf{d}_{0,2}$ then on the left side of (*) the coefficient of $t^{\textbf{d}_{0,2}}$ in the second component is positive while on the right side it is 0.

(2) Observe another fact: for each $b_{\textbf{d}_j},b_{\textbf{d}_k},b_{\textbf{d}_l},(-t^m,t^m,0,\ldots,0)$, the highest term of the first component has a negative coefficient. Thus we can use the same method as in (1) to prove that if $b_{\max1}=\max\{\textbf{d}_{j,1}, \textbf{d}_{k,1}, \textbf{d}_{l,1}, m\}$, then $b_{\max1}=\textbf{d}_{0,1}$.

(3) For an integer $m$, $m \leq \textbf{d}_{0,1}$ and $m \geq \textbf{d}_{0,2}$ are contradictory to each other because $\textbf{d}_{0,2} > \textbf{d}_{0,1}$. So
in (*) the term $(-t^m,t^m,0,\ldots,0)$ cannot appear.

(4) For each $b_{\textbf{d}_j},b_{\textbf{d}_k},b_{\textbf{d}_l}$, the lowest term of the first component has a positive coefficient. So let $b_{\min1}=\min\{\textbf{d}_{j,0}, \textbf{d}_{k,0}, \textbf{d}_{l,0}\}$, then $b_{\min1}=\textbf{d}_{0,0}$.

(5) Observe the fact that $\textbf{d}_{0,1}=\textbf{d}_{0,0}+1$, $\textbf{d}_{j,1}=\textbf{d}_{j,0}+1$, $\textbf{d}_{k,1}=\textbf{d}_{k,0}+1$, $\textbf{d}_{l,1}=\textbf{d}_{l,0}+1$ for any $j,k,l$. This, together with (2) and (4) implies that $\textbf{d}_{0,0}=\textbf{d}_{j,0}=\textbf{d}_{k,0}=\textbf{d}_{l,0}$ for any $j,k,l$.

(6) Apply $L_2^{-1}$ to (*) to get
\[a_{\textbf{d}_0}=\sum_{j \in J} q_ja_{\textbf{d}_j}+\sum_{k \in K} q_ka_{\textbf{d}_k}+\sum_{l \in L} q_la_{\textbf{d}_l}.
\tag{**}\]
The second entry of $a_{\textbf{d}}$ is $(1-t)^{-n}\pi_{\textbf{d}}(t)$. If the length of $\textbf{d}$ is $s$, then the order of zero at $t=1$ of $\pi_{\textbf{d}}(t)$ is $s$, so the order of pole at $t=1$ of $(1-t)^{-n}\pi_{\textbf{d}}(t)$ is $n-s$; as $1 \leq s \leq n$, $(1-t)^{-n}\pi_{\textbf{d}}(t)$ is a Laurent polynomial if and only if $n=s$, and if $n \neq s$, $(1-t)^{-n}\pi_{\textbf{d}}(t)$ is the second entry of a multiple of an Ext-table of a module, hence all the coefficients are positive. So in (**) the term $q_ka_{\textbf{d}_k}$ and $q_la_{\textbf{d}_l}$ does not appear, otherwise the second exponent of the right side is a power series with infinitely many terms with positive coefficients, while the second exponent of the left side is a Laurent polynomial, which is a contradiction.

(7) We get that in (*), $$b_{\textbf{d}_0}=\sum_{j \in J} q_jb_{\textbf{d}_j}.$$
All the elements are in $A_2$, so they are of the form $b_{\textbf{d}_j}=d_{j,0}<d_{j,0}+1<d_{2,j}<\ldots<d_{2,j}+n-2$. Also by (5) all the $d_{j,0}$ are equal to $d_{0,0}$. But in this case the second entry of $b_{\textbf{d}_j}$, which is $\alpha_{\textbf{d}_j}'(t)$, only has nonzero entries in $d_{j,2},\ldots,d_{j,2}+n-2$, so all these $\alpha_{\textbf{d}_j}'(t)$'s are linearly independent, which implies that all the $b_{\textbf{d}_j}$'s are linearly independent. Therefore, the decomposition is trivial.
\end{proof}
We see when $n > 2$, $A_3 \neq \emptyset$, and this implies:
\begin{proposition}\label{4.16}Any element in $A_3$ is not a positive linear combination of elements in $A_1 \cup A_2 \cup A_4$.
\end{proposition}
\begin{proof}For elements in $A_4$ the second component is 0. For elements in $A_1$ the second component is $(t^d-t^{d+1})/(1-t)^n$. It has a pole at $t=1$ of order $n-1$, and $\lim_{t \to 1}(1-t)^{n-1}(t^d-t^{d+1})/(1-t)^n = 1 > 0$. For elements in $A_4$ the second component is $\pi_{\textbf{d}}(t)/(1-t)^n$ which is regular at $t=1$. So for every linear combination of elements in $A_1 \cup A_4$ the second component is regular at $t=1$; for every linear combination of elements in $A_1 \cup A_2 \cup A_4$ where an element in $A_2$ appears, the second component of this sum is a series $f(t)$ which has a pole of order $n-1$ such that $\lim_{t \to 1}(1-t)^{n-1}f(t) > 0$. But for an element in $A_3$ the second component has a pole at $t=1$ of order $n-s$ where $2 \leq s \leq n-1$, so it cannot be a positive linear combination of elements in $A_1 \cup A_2 \cup A_4$.
\end{proof}
Proposition \ref{4.14}, \ref{4.15} and \ref{4.16} describe the cone of Ext-tables. By the local duality, they also give a description of the cone of local cohomology tables. 
\begin{definition}\label{definition-BMB'M}
For $0 \leq e \leq n$, let $S_e=k[x_1,\ldots,x_e]$ and view it as an $R$-algebra via projection $\pi: R \to S_e \cong R/(x_{e+1},\ldots,x_n)$.
\begin{enumerate}
\item $B_{d,e}$ is the set of degree sequences $\{(d_0,d_0+1),d_0 \in \mathbb{Z}\}\cup\{(d_0,d_0+1,d_2,d_2+1,\ldots,d_2+e-2), d_0,d_2 \in \mathbb{Z}, d_2 \geq d_0+2\}$;
\item $B'_{d,e}$ is the set of degree sequences $\{(d_0,d_0+1,d_2,d_2+1,\ldots,d_2+s-2),d_0,d_2 \in \mathbb{Z}, d_2 \geq d_0+2, 2 \leq s \leq e\}$;
\item $B_{M,e}$ be the following set of isomorphic classes of finitely generated $S_e$-modules viewed as $R$-modules: $M$ is in $B_{M,e}$ if and only if either it is free over $S_e$ of rank 1, or when viewing $M$ as an $S_e$-module, its projective dimension is 1, it does not have an $S_e$-summand, and $\Tr(M)$ is pure of type $\textbf{d}$ for $\textbf{d} \in B_{d,e}$.
\item Define $B'_{M,e}$ similarly as $B_{M,e}$ where we replace $B_{d,e}$ by $B'_{d,e}$.
\end{enumerate}
\end{definition} 
\begin{theorem}\label{4.17}Let $R$ be a polynomial ring of dimension $n \geq 2$. Let $C$ be the cone of local cohomology tables of modules of projective dimension at most $1$. Then:
\item (1) $\{H^{\bullet}(M), M \in B'_{M,n}\}$ generates $C$.
\item (2) $\{H^{\bullet}(M), M \in B_{M,n}\}$ is the vertex set of $C$.
\item (3) If $n > 2$, not every element in $C$ is a positive linear combination of the extremal rays.
\end{theorem}
\begin{theorem}\label{theorem-all local cohomology tables}
Let $R$ be a polynomial ring of dimension $n \geq 3$. Let $C_H$ be the cone of local cohomology tables of all finitely generated graded $R$-modules. Then $C_H$ is not generated by its vertices.
\end{theorem}
\begin{proof}
Let $C$ be the cone of local cohomology tables of modules of projective dimension at most $1$. Then $C \subset C_H$, and $x \in C_H$ is in $C$ if and only if all columns of $x$ vanish except for the last two columns which represent the $(n-1)$-th and $n$-th local cohomology. All elements in $C_H$ have nonnegative entries. Thus, if an element in $C$ decomposes in $C_H$, then this decomposition must lie in $C$, so a vertex of $C$ is also a vertex of $C_H$ in $C$. Since $n \geq 3$, we can choose $x \in C$ that does not decompose into vertices of $C$. But any decomposition of $x$ in $C_H$ also lies in $C$, so $x$ does not decompose into vertices of $C_H$.
\end{proof}
\begin{corollary}Let $e$ be an integer with $0 \leq e \leq n$. Let $C_e$ be the cone generated by local cohomology tables of modules $M$ with $\dim(M) \leq e$, \textup{depth}$M \geq e-1$. View $S_e$-modules as $R$-modules via $\pi$. Then for $e \geq 2$:
\item (1) $\{H^{\bullet}(M), M \in B'_{M,e}\}$ generates $C_e$.
\item (2) $\{H^{\bullet}(M), M \in B_{M,e}\}$ is the vertex set of $C_e$.
\item (3) If $n > 2$, not every element in $C_e$ is a positive linear combination of the extremal rays.
\end{corollary}
\begin{proof}Use Theorem \ref{4.17} and Lemma \ref{3.7}.
\end{proof}
Finally, we can describe the cone $C_s$ of local cohomology tables of saCM modules; by Proposition \ref{3.6} every local cohomology table of an saCM module decomposes into that of its dimension factors, which are almost Cohen-Macaulay, so $C_s=\sum^n_{i \geq 0}C_i$. By results in \cite{de2021decomposition}, $C_0 \subset C_1$, and $C_1$ is generated by its vertex set $\{k(a),k[x](a), a \in \mathbb{Z}\}$.
\begin{definition}\label{definition-AM,A'M}
Define $A'_M=\cup_{2 \leq e \leq n}B'_{M,e}\cup\{k(a),k[x](a), a \in \mathbb{Z}\}$, $A_M=\cup_{2 \leq e \leq n}B_{M,e}\cup\{k(a),k[x](a), a \in \mathbb{Z}\}$, where $B'_{M,e},B_{M,e}$ follows from Definition \ref{definition-BMB'M}.    
\end{definition}
Note that the $R$-module $S_e$ lies in both $B_{M,e}$ and $B_{M,e+1}$; the reason is that $S_e$ can be viewed as an $S_{e+1}$-module and in this case $S_e$ is pure with degree sequence $d_0=0<d_1=1$ and its projective dimension is 1. We have the following description of $C_s$:
\begin{theorem}\label{4.19}
\item (1) $\{H^{\bullet}(M), M \in A'_M\}$ generates $C_s$.
\item (2) $\{H^{\bullet}(M), M \in A_M\}$ is the vertex set of $C_s$.
\item (3) If $n > 2$, not every element in the cone $C_s$ is a positive linear combination of the extremal rays.
\end{theorem}
\begin{proof}(1) is trivial; the union of the generating sets of cones generates the sum of the cones. For (2) and (3), we pick an element in the generating set $\{H^{\bullet}(M_i), M_i \in A'_M\}$. If it is of the form $k(a)$ or $k[x](a), a \in \mathbb{Z}$ then it is already extremal. If it is of the form $H^{\bullet}(M_i), M_i \in B'_{M,e}$, then this local cohomology table only has two nonvanishing terms, that is, $H^{e}_{\mathfrak{m}}(M_i)$ and $H^{e-1}_{\mathfrak{m}}(M_i)$. So if it decomposes into some tables of the form $H^{\bullet}(M_i), M_i \in A'_M$ then these tables also have zero local cohomologies except for the $e$-th and the $(e-1)$-th local cohomology, which implies that $M_i \in B'_{M,e}$. So a decomposition of a generator in $C_e$ also lies in $C_e$, so a vertex of $C_s$ which lies in $C_e$ is also a vertex in $C_e$, which implies (2) and (3).
\end{proof}

\section{The $\Gamma$ functor}\label{section-Gamma functor}
Let $M$ be a finitely generated graded $R$-module. Recall that the module of global sections of $M$ is
$$\Gamma(M)=\oplus_{t \in \mathbb{Z}}H^{0}(\Proj(R), \tilde{M}(t)).$$
We can view $\Gamma$ as a functor from the category of graded $R$-modules to itself. One might hope that it maps the category of finitely generated graded $R$-modules to itself, but this is not true in general. However, if we focus on the problem of the decomposition of local cohomology tables and apply Lemma \ref{3.3} and Lemma \ref{3.4}, we may always assume that $\depth(M) \geq 1$, and the maximal submodule of dimension at most 1 is $M_1=0$. 
\begin{assumption}\label{assumptions-depth1,noM1}
We say $M$ satisfies this assumption if $\depth(M) \geq 1$ and the maximal submodule of dimension at most $1$ is $M_1=0$.    
\end{assumption}
We see if $\depth(M) \geq 2$ then $M$ satisfies Assumption \ref{assumptions-depth1,noM1}. Also, by equation \ref{equation-relation of Gamma}, $\depth(M) \geq 1$ if and only if $M$ embeds into $\Gamma(M)$ through the natural map, and $\depth(M) \geq 2$ if and only if $M = \Gamma(M)$ through this embedding.
\begin{proposition}\label{5.1}Assume $M$ satisfies Assumption \ref{assumptions-depth1,noM1}, then $H^1_{\mathfrak{m}}(M)$ has finite length.
\end{proposition}
\begin{proof}We may assume $\depth(M)=1$, otherwise $H^1_{\mathfrak{m}}(M)=0$. The condition $l(H^1_{\mathfrak{m}}(M)) < \infty$ is equivalent to $l(\Ext^{n-1}_R(M,R)) < \infty$ by local duality. Since the module $\Ext^{n-1}_R(M,R)$ is finitely generated, this module has finite length if and only if Ext$^{n-1}_R(M,R)_{\mathfrak{p}}=0$, $\forall$ ht ${\mathfrak{p}}=n-1$, which just means Ext$^{n-1}_{R_{\mathfrak{p}}}(M_{\mathfrak{p}},R_{\mathfrak{p}})=0$, $\forall$ ht ${\mathfrak{p}}=n-1$. Now the ring $R_{\mathfrak{p}}$ is a regular ring of dimension $n-1$, so apply the local duality on $R_{\mathfrak{p}}$ to get the equivalent condition $H^0_{{\mathfrak{p}}R_{\mathfrak{p}}}(M_{\mathfrak{p}})=0$, $\forall$ ht $\mathfrak{p}=n-1$. Equivalently, $\mathfrak{p} \notin $Ass$(M_{\mathfrak{p}})$, $\forall$ ht ${\mathfrak{p}}=n-1$, or $\mathfrak{p} \notin$ Ass$M$, $\forall$ ht $\mathfrak{p}=n-1$. This is true if and only if $M$ has no submodule of dimension 1.
\end{proof}
Below are some characterizations of the functor $\Gamma$.
\begin{proposition}[Universal property]Denote the natural map $M \to \Gamma(M)$ by $i$. Let $M$ and $N$ be two finitely generated graded $R$-modules, where $\depth(M)\geq 1$, $\depth(N)\geq 2$. Suppose $f:M \to N$ is an embedding, then there exists a unique embedding $f':\Gamma(M) \to N$ such that $f=f'i$.
\end{proposition}
\begin{proof}$\Gamma$ is left exact because sheafification, tensoring with $\mathcal{O}_{\mathbb{P}^{n-1}}(t)$ and $H^0$ are all left exact. Also, when depth$N \geq 2$, $\Gamma(N)=N$. So an embedding of modules $f:M \to N$ induces another embedding $\Gamma(f):\Gamma(M) \to N$. Let $\Gamma(f)=f'$. Suppose conversely we have $f=f'i$ for some embedding $f':\Gamma(M) \to N$, then $\Gamma(f)=\Gamma(f')\Gamma(i)$, but $\Gamma(f')=f'$ and $\Gamma(i) = \textup{id}_{\Gamma(M)}$, hence $f'=\Gamma(f)$ is unique.
\end{proof}
\begin{proposition}\label{5.3}Let $M$ and $N$ be two finitely generated graded $R$-modules such that $M$ embeds into $N$, $\depth(M) \geq 1$ and $\depth(N) \geq 2$. If $l(N/M)<\infty$, then $N=\Gamma(M)$.
\end{proposition}
\begin{proof}By the universal property $\Gamma(M)$ embeds into $N$. If $\Gamma(M) \neq N$, then by the depth lemma, $N/\Gamma(M)$ has depth at least 1, hence $l(N/\Gamma(M))=\infty$, hence $l(N/M)=\infty$, which is a contradiction.
\end{proof}
\begin{corollary}Let $M$ and $N$ be two finitely generated graded $R$-modules such that $M$ embeds into $N$, $\depth(M) \geq 1$ and $\depth(N) \geq 2$. Let $M^{sat}=M:_{N}\mathfrak{m}^{\infty}$. Then $\Gamma(M) = M^{sat}$.
\end{corollary}
\begin{proof}By construction, $H^0_{\mathfrak{m}}(N/M)=M^{sat}/M$, $(N/M)/H^0_{\mathfrak{m}}(N/M)=N/M^{sat}$, so $\depth(N/M^{sat}) \geq 1$. And $\depth(N) \geq 2$, hence we can apply the depth lemma to get $\depth(M^{sat}) \geq 2$, and $M^{sat}/M$ is of finite length. By Proposition \ref{5.3}, $\Gamma(M) = M^{sat}$.
\end{proof}
\begin{corollary}\label{5.5}Assume $M$ satisfies Assumption \ref{assumptions-depth1,noM1}. Then $\Gamma(M)$ is also finitely generated.
\end{corollary}
\begin{proof}If $\depth(M) > 0$, then $H^0_{\mathfrak{m}}(M)=0$, so $0 \to M \to \Gamma(M) \to H^1_{\mathfrak{m}}(M) \to 0$ is exact. Now $M$ is finitely generated, and $H^1_{\mathfrak{m}}(M)$ is of finite length by Proposition \ref{5.1}, hence $H^1_{\mathfrak{m}}(M)$ is also finitely generated, so $\Gamma(M)$ is also finitely generated.
\end{proof}
In summary, to find a decomposition of $H^{\bullet}(M)$ for a general module $M$, we can consider two exact sequences $0 \to H^0_{\mathfrak{m}}(M) \to M \to M/H^0_{\mathfrak{m}}(M) \to 0$ and $0 \to M_1 \to M \to M/M_1 \to 0$. The long exact sequences of local cohomology both have $0$ connecting map, so they induce decompositions of the local cohomology table of $M$. Hence, we can assume $M$ satisfies Assumption \ref{assumptions-depth1,noM1}. In this case by Corollary \ref{5.5}, $\Gamma(M)$ is finitely generated, of depth at least 2 which contains $M$ such that $H^{\bullet}(M)$ is equal to $H^{\bullet}(\Gamma(M))$ at position $i \geq 2$ and equal to the Hilbert function of $\Gamma(M)/M$ at position 1 which has finite length. So we need to study the local cohomology tables of modules of depth at least 2, and their quotients of finite length.

\section{Decomposition in dimension 3}\label{section-decomposition dim3}
In this section, we assume $M$ satisfies Assumption \ref{assumptions-depth1,noM1}, and analyze whether the decomposition of $H^{\bullet}(\Gamma(M))$ induces that of $H^{\bullet}(M)$. In dimension 2, this is the case, but things get complicated in dimension 3. From now on, we assume $n$ = 3, that is, $R$ is a polynomial ring over 3 variables.
Let $M, \Gamma$ be two finitely generated graded $R$-modules such that $M \subset \Gamma$. Take another submodule $F$ of $\Gamma$. Then we have a $3 \times 3$ exact diagram
\begin{center}\label{diagram 2}
\leavevmode
\xymatrix@R=1em{
  & & 0         \ar[d]& 0        \ar[d]& 0          \ar[d]& \\
C_1\ar[d]^f&0  \ar[r]& M\cap F   \ar[r]\ar[d]& M        \ar[r]\ar[d]& M/M\cap F  \ar[r]\ar[d]& 0\\
C_2\ar[d]^{f'}&0  \ar[r]& F         \ar[r]\ar[d]& \Gamma   \ar[r]\ar[d]& \Gamma/F   \ar[r]\ar[d]& 0\\
C_3&0  \ar[r]& F/M\cap F \ar[r]\ar[d]& \Gamma/M \ar[r]\ar[d]& \Gamma/(M+F) \ar[r]\ar[d]& 0.\\
  & & 0         & 0        & 0          & \\
 & & D_1\ar[r]^g & D_2 \ar[r]^{g'} & D_3 & \\
}

Diagram 1
\end{center}
This diagram contains 3 horizontal short exact sequences, 3 vertical short exact sequences, and 4 morphisms between these 6 exact sequences. Denote the 3 horizontal short exact sequences from top to bottom by $C_1,C_2,C_3$ and 3 vertical ones from left to right $D_1,D_2,D_3$. The four morphism are $f:C_1 \to C_2$, $f':C_2 \to C_3$, $g:D_1 \to D_2$ and $g':D_2 \to D_3$. Applying the functors $\{H^i_{\mathfrak{m}}(\bullet)\}_{0 \leq i \leq 3}$ to these 4 morphisms of complexes, we get 4 chain maps between long exact sequences which consist of 12 $R$-linear maps between local cohomology modules. We use $H(C_i),H(D_i),1 \leq i \leq 3; H(f),H(f'),H(g),H(g')$ for the induced long exact sequences and chain maps. The maps in $H(f),H(f'),H(g),H(g')$ are denoted by $f_i,f'_i,g_i,g'_i, 1 \leq i \leq 12$ respectively. For example, $H(f):H(C_1) \to H(C_2)$ is given by the following chain map $f_i, 1 \leq i \leq 12$. We can get $f'_i,g_i,g'_i$ similarly.

\xymatrix@C=2em{
0 \ar[r]& H^0_{\mathfrak{m}}(M\cap F) \ar[r]\ar[d]^{f_1}& H^0_{\mathfrak{m}}(M) \ar[r]\ar[d]^{f_2}& H^0_{\mathfrak{m}}(M/M\cap F) \ar[r]\ar[d]^{f_3}& H^1_{\mathfrak{m}}(M\cap F) \ar[d]^{f_4}\\
0 \ar[r]& H^0_{\mathfrak{m}}(F) \ar[r]& H^0_{\mathfrak{m}}(\Gamma) \ar[r]& H^0_{\mathfrak{m}}(\Gamma/F) \ar[r]& H^1_{\mathfrak{m}}(F) \\
\ar[r]& H^1_{\mathfrak{m}}(M) \ar[r]\ar[d]^{f_5}& H^1_{\mathfrak{m}}(M/M\cap F) \ar[r]\ar[d]^{f_6}& H^2_{\mathfrak{m}}(M\cap F) \ar[r]\ar[d]^{f_7}& H^2_{\mathfrak{m}}(M)\ar[d]^{f_8}\\
\ar[r]& H^1_{\mathfrak{m}}(\Gamma) \ar[r]& H^1_{\mathfrak{m}}(\Gamma/F) \ar[r]& H^2_{\mathfrak{m}}(F) \ar[r]& H^2_{\mathfrak{m}}(\Gamma)\\
}

\xymatrix@C=2em{
\ar[r]& H^2_{\mathfrak{m}}(M/M\cap F) \ar[r]\ar[d]^{f_9}& H^3_{\mathfrak{m}}(M\cap F) \ar[r]\ar[d]^{f_{10}}& H^3_{\mathfrak{m}}(M) \ar[r]\ar[d]^{f_{11}}& H^3_{\mathfrak{m}}(M/M\cap F) \ar[r]\ar[d]^{f_{12}}& 0\\
\ar[r]& H^2_{\mathfrak{m}}(\Gamma/F) \ar[r]& H^3_{\mathfrak{m}}(F) \ar[r]& H^3_{\mathfrak{m}}(\Gamma) \ar[r]& H^3_{\mathfrak{m}}(\Gamma/F) \ar[r]& 0.\\
}

We prove a general decomposition principle, which allows us to decompose $H^{\bullet}(M)$ as the sum of local cohomology tables of a submodule and a quotient module of $M$. Note that Assumption \ref{assumptions-depth1,noM1} implies that $\Gamma(M)$ is finitely generated.
\begin{proposition}[General decomposition principle]\label{6.1}Given a $3 \times 3$ diagram as above, and suppose it satisfies the following three conditions:
\item (a) $\depth(F) \geq 2$, $\depth(\Gamma/F) \geq 2$, $\depth(\Gamma) \geq 2$.
\item (b) $\Gamma=\Gamma(M)$. In particular, $l(\Gamma/M)<\infty$.
\item (c) The connecting homomorphism $H^2_{\mathfrak{m}}(\Gamma/F) \to H^3_{\mathfrak{m}}(F)$ is 0.

Then the following three properties hold.
\item (1) $H^{\bullet}(\Gamma)=H^{\bullet}(F)+H^{\bullet}(\Gamma/F)$.
\item (2) $F=\Gamma(M\cap F)$ and $\Gamma/F=\Gamma(M/M\cap F)$.
\item (3) $H^{\bullet}(M)=H^{\bullet}(M\cap F)+H^{\bullet}(M/M\cap F)$.
\end{proposition}
\begin{proof}
\item (1) Consider the long exact sequence $H(C_2)$. By condition (c), it decomposes into 2 short exact sequences, hence the equality holds.
\item (2) By the short exact sequence at the bottom, $\Gamma/M$ is of finite length implies that $F/M\cap F$ and $\Gamma/M+F$ are of finite length. So by condition (a), (2) holds.
\item (3) By (1) it suffices to decompose $H^1_{\mathfrak{m}}$. But by (2), $H^1_{\mathfrak{m}}(M)=\Gamma/M$, $H^1_{\mathfrak{m}}(M\cap F)=F/M\cap F$, $H^1_{\mathfrak{m}}(M/M\cap F)=\Gamma/(M+F)$. So by the short exact sequence $C_3$ at the bottom, $H^1_{\mathfrak{m}}$ also decomposes.
\end{proof}
The general decomposition principle tells us that if we can find a submodule $F$ satisfying (a)-(c), then $H^{\bullet}(M)$ decomposes.
\begin{corollary}We assume (b) of $\ref{6.1}$ holds in diagram $2$. If $\Gamma$ is decomposable, then $H^{\bullet}(M)$ decomposes nontrivially into the cohomology table of two nonzero modules.
\end{corollary}
\begin{proof}Suppose $F$ is a nonzero direct summand, then $\Gamma=F\oplus \Gamma/F$. In this case the condition (a) and (c) are satisfied, so by (3), $H^{\bullet}(M)=H^{\bullet}(M\cap F)+H^{\bullet}(M/M\cap F)$. By (3), $F, \Gamma/F \neq 0$ implies $M\cap F, M/M\cap F \neq 0$.
\end{proof}
\begin{corollary}We assume (b) of $\ref{6.1}$ holds in diagram $2$. If $H^2_{\mathfrak{m}}(M)=0$ or $H^3_{\mathfrak{m}}(M)=0$, then $H^{\bullet}(M)$ decomposes as a sum of $H^{\bullet}(N_i)$ such that $\Gamma(N_i)$ is cyclic.
\end{corollary}
\begin{proof}If $H^3_{\mathfrak{m}}(M)=0$, then $\Gamma=\Gamma(M)$ is Cohen-Macaulay of dimension 2, then it reduces to the case in \cite{CAVIGLIA2021106635}. If $H^2_{\mathfrak{m}}(M)=0$, then $H^i_{\mathfrak{m}}(\Gamma)$ is zero except for $i=3$, so $\Gamma$ is free, hence it is a direct sum of free cyclic modules, so (a) and (c) hold for any free summand $F \subset \Gamma$. And (b) holds by assumption, so by Proposition \ref{6.1} we get (1)-(3) for $M,\Gamma,F$ and $H^{\bullet}(M)=H^{\bullet}(M\cap F)+H^{\bullet}(M/M\cap F)$. We may work with either $M \cap F$ or $M/M\cap F$ whenever $F$ or $\Gamma/F$ is decomposable and free, and use (2) and (3) repeatedly. Finally, we end up with $H^{\bullet}(N_i)$ such that $H^{\bullet}(M)=\sum_i H^{\bullet}(N_i)$ and $\Gamma(N_i)$ is an indecomposable direct summand of $\Gamma$, hence must be free cyclic.
\end{proof}
\begin{remark}The above two lemmas explain why things are different in dimension 2 and dimension 3. In dimension 2, $\Gamma$ is free because $\depth(\Gamma) \geq 2$, so it reduces to the case where $\Gamma$ is free of rank 1. In dimension 3 this is not always true.
\end{remark}
Pick a module $M$ of depth 1 without a dimension 1 submodule and let $\Gamma=\Gamma(M)$. In general it is hard to find a submodule of $\Gamma$ that satisfies both (a) and (c). There are two ways to approach this. The first way is to take $F$ to be a submodule of dimension 2; and the second way is to take $F$ to be a free submodule. In the first way, (c) is satisfied but the quotient $\Gamma/F$ may violate (a) because we may have $\depth(\Gamma/F) = 1$.

We start with some lemmas on $\Gamma$ where $\Gamma$ is a general finitely generated graded $R$-module.
\begin{lemma}\label{6.5}Let $\Gamma$ be a module of depth $2$. Then the maximal submodule of dimension at most $2$ is the torsion submodule $\Tor(\Gamma)$. If $\depth(\Gamma) \geq 2$, then $\Tor(\Gamma)$ is Cohen-Macaulay of dimension $2$.
\end{lemma}
\begin{proof}The maximal submodule of dimension at most 2 is generated by all $m \in \Gamma$ such that ann$_R(m) \neq 0$, so it is $\Tor(\Gamma)$. We have an exact sequence $0 \to \Tor(\Gamma) \to \Gamma \to Q \to 0$ where $Q=\Gamma/\Tor(\Gamma)$. By assumption depth($\Gamma$)$\geq$ 2. Also, $H^0_{\mathfrak{m}}(Q)=0$ because $Q$ does not have submodule of dimension less than 2, so depth($Q$)$\geq$ 1. Therefore we have depth($\Tor(\Gamma)) \geq 2$ by the depth lemma, but dim($\Tor(\Gamma)) \leq 2$, so it is Cohen-Macaulay of dimension 2.
\end{proof}
The measurement of $\Gamma/F$ violating (a) is given by the module $H^1_{\mathfrak{m}}(Q)$. It suffices to give a description of this module.
\begin{lemma}Let $\Gamma$ be a module of depth $2$, then $\Gamma^*$ and $\Gamma^{**}$ are modules of depth at least $2$.
\end{lemma}
\begin{proof}If $\Gamma$ has a free summand $F$, then $F^*$ is a free summand of $\Gamma^*$, so it suffices to prove in the case where $\Gamma$ has no free summand. In this case $\Gamma^*$ is the second syzygy of $\Tr(\Gamma)$. Hence projdim$(\Gamma^*) \leq 3-2=1$. So depth$(\Gamma^*) \geq 2$. Replace $\Gamma$ by $\Gamma^*$, we get depth$(\Gamma^{**}) \geq 2$.
\end{proof}
\begin{proposition}\label{6.7}There is an exact sequence $0 \to \Tor(\Gamma) \to \Gamma \to \Gamma^{**} \to L \to 0$. Then $L=\Ext^2_R(\Tr(\Gamma),R)$. Let $Q=\Gamma/\Tor(\Gamma)$, then $\depth(Q) \geq 1$ and $H^1_{\mathfrak{m}}(Q)=H^0_{\mathfrak{m}}(L)$ is of finite length.
\end{proposition}
\begin{proof}The exact sequence is well-known, so we omit the proof of the exactness. Now let $Q=\Gamma/\Tor(\Gamma)$. The composition map $\Tor(\Gamma) \to \Gamma \to \Gamma^{**}$ is 0, hence we have a map $Q \to \Gamma^{**}$. This map is injective because it is a map between torsion-free modules over $R$ and it stays injective after tensoring with $K$. So there are two short exact sequences $0 \to \Tor(\Gamma) \to \Gamma \to Q \to 0$ and $0 \to Q \to \Gamma^{**} \to L \to 0$. This leads to two long exact sequences:
\begin{center}
\leavevmode
\xymatrix{
0 \ar[r]& H^1_{\mathfrak{m}}(\Tor(\Gamma))=0 \ar[r]& H^1_{\mathfrak{m}}(\Gamma)=0 \ar[r]& H^1_{\mathfrak{m}}(Q) \\
  \ar[r]& H^2_{\mathfrak{m}}(\Tor(\Gamma)) \ar[r]& H^2_{\mathfrak{m}}(\Gamma) \ar[r]& H^2_{\mathfrak{m}}(Q) \\
  \ar[r]& H^3_{\mathfrak{m}}(\Tor(\Gamma)) \ar[r]& H^3_{\mathfrak{m}}(\Gamma) \ar[r]& H^3_{\mathfrak{m}}(Q) \ar[r]&0,\\
}
\end{center}
\begin{center}
\leavevmode
\xymatrix{
0 \ar[r]& H^0_{\mathfrak{m}}(Q)=0 \ar[r]& H^0_{\mathfrak{m}}(\Gamma^{**})=0 \ar[r]& H^0_{\mathfrak{m}}(L) \\
  \ar[r]& H^1_{\mathfrak{m}}(Q) \ar[r]& H^1_{\mathfrak{m}}(\Gamma^{**})=0 \ar[r]& H^1_{\mathfrak{m}}(L) \\
  \ar[r]& H^2_{\mathfrak{m}}(Q) \ar[r]& H^2_{\mathfrak{m}}(\Gamma^{**}) \ar[r]& H^2_{\mathfrak{m}}(L)=0 \\
  \ar[r]& H^3_{\mathfrak{m}}(Q) \ar[r]& H^3_{\mathfrak{m}}(\Gamma^{**}) \ar[r]& H^3_{\mathfrak{m}}(L)=0 \ar[r]&0.\\
}
\end{center}
As in the proof of Lemma \ref{6.5}, depth$(Q) \geq$ 1. Also depth$(\Gamma^{**}) \geq$ 2, so $H^0_{\mathfrak{m}}(\Gamma^{**})=H^1_{\mathfrak{m}}(\Gamma^{**})=0$, hence $H^1_{\mathfrak{m}}(Q) \cong H^0_{\mathfrak{m}}(L)$. Note that $L$ is a finitely generated module over $R$, so $H^0_{\mathfrak{m}}(L)$ is of finite length.
\end{proof}
The following theorem explains that we cannot always get a decomposition of $H^\bullet(M)$ in the first way, and gives the condition when such a decomposition exists. 
\begin{theorem}\label{6.8}
Let $M$ be a module satisfying Assumption \ref{assumptions-depth1,noM1}. Let $\Gamma=\Gamma(M)$ and $Q=\Gamma/\Tor(\Gamma)$. Then 
$$H^{\bullet}(M)=H^{\bullet}(\Tor(M))+H^{\bullet}(M/\Tor(M))-(0,HS(H^1_{\mathfrak{m}}(Q)),HS(H^1_{\mathfrak{m}}(Q)),0).$$ 
In particular, if $H^1_{\mathfrak{m}}(Q)=0$, then 
$$H^{\bullet}(M)=H^{\bullet}(\Tor(M))+H^{\bullet}(M/\Tor(M)).$$
\end{theorem}
\begin{proof}Take $F=\Tor(\Gamma)$ in diagram 2, then $M\cap F=\Tor(M)$, $\Gamma/F=Q$. Let $H(f):H(C_1) \to H(C_2)$ be the corresponding morphism. Then $f_i, 7 \leq i \leq 12$ are isomorphisms, and $f_6$ is surjective. Note that in this case $H^3_{\mathfrak{m}}(F)=H^3_{\mathfrak{m}}(M\cap F)=0$. Now eliminate all the 0's and get a diagram:
\begin{center}
\leavevmode
\xymatrix{
0 \ar[r]& H^1_{\mathfrak{m}}(\Tor(M)) \ar[r]\ar[d]^{f_4}& H^1_{\mathfrak{m}}(M) \ar[r]\ar[d]^{f_5}& H^1_{\mathfrak{m}}(M/\Tor(M)) \ar[d]^{f_6}\\
0 \ar[r]& 0 \ar[r]& 0 \ar[r]& H^1_{\mathfrak{m}}(Q) \\
}
\end{center}
\begin{center}
\leavevmode
\xymatrix{
\ar[r]&H^2_{\mathfrak{m}}(\Tor(M)) \ar[r]\ar[d]^{f_7}& H^2_{\mathfrak{m}}(M)\ar[r]\ar[d]^{f_8}&H^2_{\mathfrak{m}}(M/\Tor(M))\ar[d]^{f_9}\ar[r]& 0\\
\ar[r]&H^2_{\mathfrak{m}}(\Tor(\Gamma)) \ar[r]& H^2_{\mathfrak{m}}(\Gamma)\ar[r]&H^2_{\mathfrak{m}}(Q)\ar[r]& 0\\
}
\end{center}
plus isomorphisms $H^3_{\mathfrak{m}}(M)=H^3_{\mathfrak{m}}(\Gamma)=H^3_{\mathfrak{m}}(M/\Tor(M))=H^3_{\mathfrak{m}}(Q)$. So the kernel of the map $H^2_{\mathfrak{m}}(\Tor(M)) \to H^2_{\mathfrak{m}}(M)$ is isomorphic to $H^1_{\mathfrak{m}}(Q)$. This leads to 3 equations:
$$HS(H^3_{\mathfrak{m}}(M))=HS(H^3_{\mathfrak{m}}(\Tor(M))),$$
$$HS(H^2_{\mathfrak{m}}(M))=HS(H^2_{\mathfrak{m}}(M/\Tor(M)))+HS(H^2_{\mathfrak{m}}(\Tor(M)))-HS(H^1_{\mathfrak{m}}(Q)),$$ $$HS(H^1_{\mathfrak{m}}(M))=HS(H^1_{\mathfrak{m}}(M/\Tor(M)))+HS(H^1_{\mathfrak{m}}(\Tor(M)))-HS(H^1_{\mathfrak{m}}(Q)).$$
Equivalently, we have
$$H^{\bullet}(M) = H^{\bullet}(\Tor(M))+H^{\bullet}(M/\Tor(M))-e,$$
where $e=(0,HS(H^1_{\mathfrak{m}}(Q)),HS(H^1_{\mathfrak{m}}(Q)),0).$
\end{proof}
The above theorem shows that vanishing of the term $H^1_{\mathfrak{m}}(Q)$ implies that $H^{\bullet}(M)$ decomposes as a sum of two local cohomology tables $H^{\bullet}(\Tor(M))+H^{\bullet}(M/\Tor(M))$, where $Q=\Gamma/\Tor(\Gamma)$. Here $\Tor(M)$ is a submodule of $M$ of dimension 2. In general $H^1_{\mathfrak{m}}(Q)$ does not vanish, so we do not always get a decomposition. Let $L=\textup{Ext}^2_R(\Tr(\Gamma),R)$. By Proposition \ref{6.7} we see $H^1_{\mathfrak{m}}(Q) \cong H^0_{\mathfrak{m}}(L)$, and the next two propositions show when it vanishes and how to calculate it in terms of $\Gamma$.
\begin{proposition}Let $M$ be a module satisfying Assumption \ref{assumptions-depth1,noM1}, $\Gamma=\Gamma(M)$, $Q=\Gamma/\Tor(\Gamma)$ and $L=\textup{Ext}^2_R(\Tr(\Gamma),R)$. Then $H^1_{\mathfrak{m}}(Q)=0$ if and only $L = 0$ or $L$ is Cohen-Macaulay of dimension $1$.
\end{proposition}
\begin{proof}For a finitely generated module $\Gamma$, dim$L \leq 3-2 = 1$. So if $L \neq 0$, then $H^0_{\mathfrak{m}}(L)=0$ if and only if depth$L \geq$ 1, if and only if $L$ is a Cohen-Macaulay module of dimension 1.
\end{proof}
\begin{proposition}\label{6.10}
Let $\Gamma$ be a finitely generated module over $R$ without a free summand. Let $N=\Tr(\Gamma)$, $N_1$ be the maximal submodule of dimension at most $1$, $Q'=N/N_1$, and $L = \textup{Ext}^2_R(N,R)$. Let $\Gamma' = \Gamma(Q')$. Then $H^0_{\mathfrak{m}}(L)$ = \textup{Hom}$_R(\Gamma'/Q',E)$, where $E$ is the injective hull of $k$.
\end{proposition}
\begin{proof}We may assume depth$N \geq$ 1 because Ext$^2_R(N,R) = $ Ext$^2_R(N/H^0_{\mathfrak{m}}(N),R)$. In this case $N_1$ is Cohen-Macaulay of dimension 1. The short exact sequence $0 \to N_1 \to N \to Q' \to 0$ induces a long exact sequence:
\begin{center}
\leavevmode
\xymatrix{
0 \ar[r]& H^0_{\mathfrak{m}}(N_1) \ar[r]& H^0_{\mathfrak{m}}(N) \ar[r]& H^0_{\mathfrak{m}}(Q')=0 \\
  \ar[r]& H^1_{\mathfrak{m}}(N_1) \ar[r]& H^1_{\mathfrak{m}}(N) \ar[r]& H^1_{\mathfrak{m}}(Q') \\
  \ar[r]& H^2_{\mathfrak{m}}(N_1)=0 \ar[r]& H^2_{\mathfrak{m}}(N) \ar[r]& H^2_{\mathfrak{m}}(Q'). \\
}
\end{center}
So we have an exact sequence $0 \to H^1_{\mathfrak{m}}(N_1) \to H^1_{\mathfrak{m}}(N) \to H^1_{\mathfrak{m}}(Q') \to 0$. By local duality, $0 \to $ Ext$^2_R(Q',R) \to L \to $ Ext$^2_R(N_1,R) \to 0$ is exact. Now dim$L \leq 1$, Ext$^2_R(N_1,R)$ is Cohen-Macaulay of dimension 1, and $H^1_{\mathfrak{m}}(Q')$ is of finite length, hence Ext$^2_R(Q',R)$ is of finite length. This means that Ext$^2_R(Q',R) = H^0_{\mathfrak{m}}(L)$. Finally by local duality Ext$^2_R(Q',R)=$ Hom$_R(H^1_{\mathfrak{m}}(Q'),E)$ and $H^1_{\mathfrak{m}}(Q') = \Gamma'/Q'$, so we are done.
\end{proof}
Theorem \ref{6.8} describes a way to decompose a local cohomology table of a module $M$ of dimension 3 using a submodule $\Tor(M)$ of dimension 2. There is another way to decompose $H^{\bullet}(M)$, which is induced by a free submodule $F \subset \Gamma$. Note that if $\Gamma$ does not have a free submodule, then $\dim(\Gamma) \leq 2$, $\dim(M) \leq 2$, and the decomposition of the local cohomology table of $M$ is known by \cite{CAVIGLIA2021106635}. So we may always assume that $\Gamma$ has a free submodule. In this case (a) is automatically satisfied because the depth lemma implies that if depth$\Gamma \leq 2$ and depth$F \leq 3$, then depth$\Gamma/F \geq 2$. Hence $F$ may only violate (c) in the general decomposition principle. We observe that the only possibly nonzero connection map is $H^2_{\mathfrak{m}}(\Gamma/F) \to H^3_{\mathfrak{m}}(F)$, and $H^2_{\mathfrak{m}}(\Gamma/F) \to H^3_{\mathfrak{m}}(F)$ is 0 if and only if $H^3_{\mathfrak{m}}(F) \to H^3_{\mathfrak{m}}(\Gamma)$ is injective, if and only if $\Gamma^* \to F^*$ is surjective. But $F^*$ is free, so this means $F^*$ is a free summand of $\Gamma^*$.
\begin{proposition}\label{6.12}Let $\Gamma$ be a module of dimension $3$ and depth at least $2$. Suppose $\Gamma^*$ has a free summand $G$ and $\Ext^2_R(\Tr(\Gamma),R) = 0$. Then $F = G^* \subset \Gamma$ is a free summand, and the pair $(F, \Gamma)$ satisfies condition ($3$) of Proposition \ref{6.1}.
\end{proposition}
\begin{proof}Since $G$ is a free summand of $\Gamma^*$, $F=G^*$ is a free summand of $\Gamma^{**}$, so $\Gamma^{**}$ surjects onto $F$. Now since $\Ext^2_R(\Tr(\Gamma),R)=0$, $\Gamma$ surjects onto $\Gamma^{**}$, so $\Gamma$ surjects onto $F$, but $F$ is projective, hence $F$ is a free summand. The map $\Gamma^* \to F^*$ induced by the inclusion $F \to \Gamma$ is just the projection onto the summand $G$ which is surjective. This means that the pair $(F, \Gamma)$ satisfies (c) of Proposition \ref{6.1}.
\end{proof}
\begin{proposition}\label{6.13}Let $\Gamma$ be a module of dimension $3$ and depth at least $2$. Then $\Gamma^*$ does not have a free summand if and only if $\Gamma^*=\Tr(L')$ for a module $L'$ of finite length.
\end{proposition}
\begin{proof}By the previous proposition projdim$(\Gamma^*) \leq 1$. So $\Gamma^*$ does not have a free summand if and only if $\Gamma^*=\Tr(L')$ for $L'=\Tr(\Gamma^*)$. But $\Gamma^*$ is the second syzygy of $N=\Tr(\Gamma)$, hence $L'=\textup{Ext}^3_R(N,R)$, and this module has finite length.
\end{proof}
By Proposition \ref{6.12} and \ref{6.13} we immediately have:
\begin{theorem}\label{6.14}
Let $M$ be an $R$-module satisfying Assumption \ref{assumptions-depth1,noM1}, $\Gamma=\Gamma(M)$. Assume $\Ext^2_R(\Tr(\Gamma),R)=0$, and $\Gamma^* \neq \Tr(L')$ for any module $L'$ of finite length. Then there exists a free submodule $F \subset \Gamma$ such that $H^{\bullet}(M)=H^{\bullet}(M\cap F)+H^{\bullet}(M/M\cap F)$.
\end{theorem}
Since $l(F/M\cap F)<\infty$, we know dim$(M\cap F)$ = dim$F$ = 3, so $H^{\bullet}(M)$ is the sum of $H^{\bullet}(M\cap F)$ and $H^{\bullet}(M/M\cap F)$ where $M\cap F$ is a submodule of $M$ of dimension 3.

In conclusion, for a finitely generated graded module $M$ of dimension 3, $H^{\bullet}(M)$ is decomposable in two cases; in Theorem \ref{6.8} a submodule of $M$ of dimension 2 induces a decomposition and in Theorem \ref{6.14} a submodule of $M$ of dimension 3 induces a decomposition.

\section*{Acknowledgements}
The author would like to thank Giulio Caviglia for introducing this problem and providing references, and the referee for helpful suggestions. The author is supported by the Ross-Lynn Research Scholar Fund of Purdue University.

\bibliographystyle{plain}
\bibliography{refsDLCT2}

\end{document}